\DeclareSymbolFont{AMSb}{U}{msb}{m}{n}
\documentclass[reqno, a4paper, 11pt]{amsart}

\usepackage[T1]{fontenc}
\usepackage[utf8]{inputenc}
\usepackage[foot]{amsaddr}
\usepackage{amsthm,amsfonts,amssymb,amsxtra,appendix,bookmark,euscript,delarray,dsfont,mathrsfs,stackrel,xifthen,mathtools}
\usepackage[normalem]{ulem}
\usepackage[shortlabels]{enumitem}
\usepackage{xcolor}
\usepackage{latexsym,amssymb}
\usepackage{mathrsfs}
\usepackage{longtable}
\usepackage{lmodern}
\usepackage{microtype}
\usepackage{tikz}
\usetikzlibrary{cd}
\usepackage{adjustbox}
\usepackage[bbgreekl]{mathbbol}
\usepackage{comment}
\usepackage{calc}% http://ctan.org/pkg/calc
\usepackage{centernot}
\usepackage[includeheadfoot,margin=2.54cm]{geometry}

\usepackage[nodayofweek]{datetime}
\usepackage{tikz-cd}
\usepackage{adjustbox}

\hypersetup{hidelinks}

\DeclareSymbolFontAlphabet{\mathbbm}{bbold}
\DeclareSymbolFontAlphabet{\mathbb}{AMSb}%

\DeclareRobustCommand{\SkipTocEntry}[5]{}

\makeatletter
\newtheorem*{rep@theorem}{\rep@title}
\newcommand{\newreptheorem}[2]{%
	\newenvironment{rep#1}[1]{%
		\def\rep@title{#2 \ref{##1}}%
		\begin{rep@theorem}}%
		{\end{rep@theorem}}}
\makeatother

\theoremstyle{plain}
\newtheorem{theorem}{Theorem}[section]
\newtheorem{lemma}[theorem]{Lemma}
\newreptheorem{lemma}{Lemma}
\newtheorem*{lemma*}{Lemma}
\newtheorem{corollary}[theorem]{Corollary}
\newtheorem{proposition}[theorem]{Proposition}

\newtheorem*{conjecture*}{Conjecture}
\newtheorem{assumption}[theorem]{Assumption}

\theoremstyle{definition}
\newtheorem{definition}[theorem]{Definition}

\newreptheorem{example}{Example}
\theoremstyle{remark}
\newtheorem{remark}[theorem]{Remark}

\DeclareMathOperator{\Lip}{Lip}

\def\bq{\begin{eqnarray}}
	\def\eq{\end{eqnarray}}
\def\bqq{\begin{eqnarray*}}
	\def\eqq{\end{eqnarray*}}

\def\epsilon{\varepsilon}

\newcommand\1{{\ensuremath {\mathds 1} }}
\newcommand\pscal[1]{{\ensuremath{\left\langle #1 \right\rangle}}}

\newcommand{\bsnorm}[2][]{%
	\ifthenelse{\isempty{#1}}%
	{{\ensuremath{|\! |\! |  #2 |\! |\! |_{\beta,s}}}}% if #1 is empty
	{{\ensuremath{|\! |\! |  #2 |\! |\! |_{\beta,s}^#1}}}% if #1 is not empty
}

\def\bN{\mathbb{N}}
\def\bR{\mathbb{R}}

\def\bP{\mathbb{P}} 
\def\bE{\mathbb{E}}

\def\cL {\mathcal{L}}

\def\cA{\mathcal{A}}

\def\cP{\mathcal{P}}
\def\cM{\mathcal{M}}

\allowdisplaybreaks

\DeclareMathAlphabet{\mathup}{OT1}{\familydefault}{m}{n}
\newcommand{\dx}[1]{\mathop{}\!\mathup{d} #1}
\newcommand{\pderiv}[3][]{\frac{\mathop{}\!\mathup{d}^{#1} #2}{\mathop{}\!\mathup{d} #3^{#1}}}
\def\d{\dx}

\newcommand{\eps}{\varepsilon}
\DeclarePairedDelimiter{\abs}{\lvert}{\rvert}

\DeclarePairedDelimiter{\bra}{(}{)}
\DeclarePairedDelimiter{\pra}{[}{]}
\DeclarePairedDelimiter{\set}{\{}{\}}
\DeclarePairedDelimiter{\skp}{\langle}{\rangle}

\makeatletter
\newcommand{\customlabel}[2]{%
   \protected@write \@auxout {}{\string \newlabel {#1}{{#2}{\thepage}{#2}{#1}{}} }%
   \hypertarget{#1}{}
}
\makeatother

\numberwithin{figure}{section}
\numberwithin{equation}{section}

\def\calA{{\mathcal A}}

\def\calP{{\mathcal P}}

 \def\bbN{{\mathbb N}} 
  \def\bbR{{\mathbb R}}

\allowdisplaybreaks

\title[Stochastic particle system convergence to CGEDG]{%
Convergence of a Stochastic Particle System to the Continuous Generalized Exchange-Driven Growth Model}

\author{Chun Yin Lam \and André Schlichting}
\email{\{chun.lam,andre.schlichting\}@uni-ulm.de}
\address{Institute for Applied Analysis, University of Ulm, Germany}
\thanks{This work is supported by the Deutsche Forschungsgemeinschaft (DFG, German Research Foundation) under Germany's Excellence Strategy EXC 2044 --390685587, Mathematics M\"unster: Dynamics--Geometry--Structure.}

\begin{document}
\begin{abstract}
The continuous generalized exchange-driven growth model (CGEDG) is a system of integro-differential equations describing the evolution of cluster mass under mass exchange. 
The rate of exchange depends on the masses of the clusters involved and the mass being exchanged. 
This can be viewed as both a continuous generalization of the exchange-driven growth model and a coagulation-fragmentation equation that generalizes the continuous Smoluchowski equation.

Starting from a Markov jump process that describes a finite stochastic interacting particle system with exchange dynamics, we prove the weak law of large numbers for this process for sublinearly growing kernels in the mean-field limit.
We establish the tightness of the stochastic process on a measure-valued Skorokhod space induced by the $1$-Wasserstein metric, from which we deduce the existence of solutions to the (CGEDG) system.
The solution is shown to have a Lebesgue density under suitable assumptions on the initial data. 
Moreover, within the class of solutions with density, we establish the uniqueness under slightly more restrictive conditions on the kernel. 
\end{abstract}

\maketitle

%\tableofcontents

\section{Introduction} 

In this work, we study a binary collision model known as the continuous generalized exchange-driven growth (CGEDG) equations which is given by the following set of integro-differential equations 
\begin{equation*}\label{eq:CGEDG}\tag{CGEDG}\begin{split}
\partial_t c (a)& = \int_0^a  \int_z^\infty K(x,a-z,z) c(x)c(a-z)   \d x   \d z  \\
&\quad - 
\int_0^a  \int_0^\infty   K(a,x,z) c(a) c(x)  \d x \d z \\
&\quad 
- \int_0^\infty \int_z^\infty K(x,a,z)c(x) c(a)  \d x \d z \\
&\quad 
+ \int_0^\infty \int_0^\infty K(a+z,x,z) c(x) c(a+z)  \d x  \d z.
 \end{split}\end{equation*}
The system describes the exchange of mass fractions $z\geq 0$ between clusters of size $x,y\geq 0$, which can be represented in the following diagram
\begin{equation*}%\label{eq:react:CGEGD}
	\set*{x} + \set*{y+z} \xrightleftharpoons[K(x+z,y,z)]{K(x,y+z,z)} 
	\set*{x+z} + \set*{y} \,.
\end{equation*}
Note that an exchange event can be described by a fragmentation of the cluster $\set*{y+z}$ into $\set*{y}$ and $\set*{z}$ and immediate coagulation of $\set*{x}$ with $\set*{z}$ to $\set*{x+z}$. 
In this sense, the system bears similarity to the classical continuous Smoluchowski coagulation-fragmentation equation~\cite{Smoluchowski1916}.

In this work, we propose a stochastic interacting particle system with exchange dynamics that converges weakly to the solution of~\eqref{eq:CGEDG}. 
The convergence result provides a weak law of large numbers for the stochastic process and it implies the existence of the weak solution to~\eqref{eq:CGEDG}. 
The proof is a tightness argument of probability measures on the Skorokhod space of measure-valued process. 
The full weak law of large numbers is proved in three steps. 
First, we provide criteria on the kernel $K$ such that the sequence of the finite particle process is tight. 
Second, under appropriate initial conditions, we show the limit measure has Lebesgue density. 
Third, under more restrictive assumptions on $K$, we show the uniqueness for solutions to~\eqref{eq:CGEDG}.

\subsection{Connections to previous works}

This model can be viewed as a generalized exchange-driven growth (EDG) model~\cite{BenNaimKrapivsky2003} with continuous cluster size. In the sense that, in the EDG model, the cluster sizes are positive integers and the kernel $(K(k,l-1))_{k,l\in\bN}$ only depends on the cluster sizes of the two clusters involved.
For the EDG model, the existence results were investigated analytically in \cite{Esenturk2018,EichenbergSchlichting2021,si2024existence}. 
The global or local in time existence of its solution depends on the growth of the kernel. 
For general kernels with at most linear growth $K(k,l)\le C_K k l$, the global existence was proved in \cite{Esenturk2018} and nonexistence for $K(k,l)\ge C_K k^{\beta}, \beta >1$.  
For superlinear growth, different regimes can be observed depending on the growth of the kernel. 
Under suitable symmetry conditions, global existence can still hold. 
In general, however, finite time or even instantaneous gelation, that is, the appearance of an infinite cluster, is expected for such type of models~\cite{Aldous1998,EscobedoMischlerPerthame2002}.
Moreover for the at most linear kernels, under the detailed balance conditions the long time convergence to equilibrium was proved in \cite{Schlichting2020,EsenturkValazquez2021}. 
In this case, there might exist a critical first moment above which there is a loss of mass to infinity in the long-time limit.

%and resembles the continuous Smoluchowski coagulation-fragmentation equation. 
A first generalization of the EDG model is provided in~\cite{barik2024discrete}, where the kernel also depends on the mass being exchanged, but the cluster sizes are still discrete. 
The continuous generalization~\eqref{eq:CGEDG} has been proposed in \cite{barik2025continuous} where its well-posedness has been proved. 
In the works cited so far on the EDG model and its generalizations, the existence of a solution was typically derived analytically from the truncation of the system towards a finite dimensional problem and then proving suitable compactness for passing to the limit in the truncation parameter.
This approach of proof traces back to~\cite{BCP86} for the related Becker-Döring equations~\cite{BeckerDoering1935}.

For the EDG model, the convergence from the particle process to the mean-field equation is given in~\cite{GrosskinskyJatuviriyapornchai2019} using a stochastic approach and in~\cite{EDG-convergence} using variational techniques.
In this work, we give a convergence proof with the stochastic approach from a particle process to~\eqref{eq:CGEDG} for a general kernel with a sublinear growth condition, which supplements the well-posedness result from~\cite{barik2025continuous}. 
For the EDG model a similar weak law of large numbers was obtained in~\cite{GrosskinskyJatuviriyapornchai2019}. In comparison, we also relax the initial second moments condition to a class of superlinear moments. The main technical difference to~\cite{GrosskinskyJatuviriyapornchai2019} is that the state space for~\eqref{eq:CGEDG} is $\cP([0,\infty))$ instead of $\cP(\bN_0)$. As $[0,\infty)$ is not countable, the diagonal sequence argument used to identify the limit point as the solution to~\eqref{eq:CGEDG} needs to be modified. 

The system~\eqref{eq:CGEDG}, like the continuous Smoluchowski coagulation-fragmentation equation, is a differential-integral equation and models the formation of clusters via exchanges of particles. The Smoluchowski equation, which also has both discrete \cite{Smoluchowski1916,Laurencot2002,white1980global,spouge1985existence} and continuous versions~\cite{laurenccot2004coalescence,aldous1999deterministic}, is different from \eqref{eq:CGEDG} because for \eqref{eq:CGEDG} the exchange rate depends on the densities of the two involved clusters both in coagulation and in fragmentation. In this sense, \eqref{eq:CGEDG} is quadratic in both directions. Furthermore, in \eqref{eq:CGEDG} the rate is allowed to vary depending on the mass of particle being exchanged. 
The theory of stochastic process convergence to the Smoluchowki equation is much more developed and is generalized to different state spaces, other than real number sizes. The process we propose for the model~\eqref{eq:CGEDG} can be seen as an analogue to the Marcus-Lushnikov process~\cite{Marcus1968,Lushnikov1978,lushnikov1978some} for the Smoluchowski coagulation-fragmentation equation for which the convergence to the Smoluchowski equation has been proved in~\cite{aldous1999deterministic,norris1999smoluchowski,norris2000cluster,FournierGiet2004}. 
We restrict ourselves to kernels with sublinear growth, which is analogous to the sublinear regime of the EDG model so that one expects global in-time existence due to the absence of gelation.
However, to the best of our knowledge, the gelation aspects for the system~\eqref{eq:CGEDG} have not yet been investigated. 
Moreover, with this particle process approach, one may investigate the gelation phenomenon analogous to the Smoluchowski equations as done in~\cite{FournierLaurencot2009, andreis2024convergence, jeon1998existence}.
 
\subsection{Setting}
We consider the evolution of a system with $N\in \bbN$ particles of mass $\epsilon>0$ in $L\in \bbN$ nodes where the nodes are fully connected. Particles at the same node may be regarded as forming a single cluster, in this description, particles in a cluster can be transferred to any other cluster. A configuration $\eta$ in the configuration space $V^{N,L,\epsilon}$ is  a vector of length $L$, with $\eta_{i}\in\{0,\epsilon,...,N\epsilon\}$ giving the mass of particles in cluster $i \in \{1,...,L\}$.
%, for which we will also call the cluster size of cluster $i$.
The evolution of particles between clusters is given as a Markov jump process of finite state space with kernel $K$  depending on the cluster sizes of the source and destination clusters, as well as the amount of particles being transferred.  The empirical cluster distribution of a configuration~$\eta$ is given by $c(k \epsilon)=C^L[\eta](k\epsilon)  = \frac{1}{L} \sum_{i=1}^L  \delta_{\eta_i, k\epsilon}$ for $k\in{0,1,...,N}$. We denote the set of possible empirical measures as $\hat{V}^{N,L,\epsilon}$. We consider the family $(\hat{V}^{N,L,\epsilon})_{N\epsilon/L \to \rho}$ as a subset of probability measure on $\bR_+=[0,\infty)$. The convergence we will show is then from the empirical measure process to the delta measure at the solution of an evolution equation as $N\epsilon /L \to \rho\ge0$, and as $N,L \to \infty$, $\epsilon \to 0$.  This may be referred to as a thermodynamic limit.

For fixed $N,L,\epsilon$, the evolution of the empirical measure of finite particle is given by the Markov process with generator $\mathcal{Q}^{N,L,\epsilon}$ defined for $G:\cP(\bR_+)\to \bbR$ by
\[\mathcal{Q}^{N,L,\epsilon}G(c) = L \sum_{k=1}^N \sum_{l=0}^{N-k}\sum_{d=1}^k \kappa^L_{\epsilon}[c](k, l, d) \bra*{G\bra[\big]{c+ L^{-1} \gamma^{k\epsilon,l\epsilon,d\epsilon}}-G(c)},\]
where
\[\kappa^L_\epsilon[c](x ,y ,z )= \frac{L}{L-1} K(x ,y ,z ) \epsilon c(x) \bra[\big]{c(y)-L^{-1} \delta_{x,y}},\]
and the change of cluster sizes after a jump is decoded by the vector 
\begin{equation}\label{eq:def:gamma}
	\gamma^{k\epsilon,l\epsilon,d\epsilon}= \delta_{(k-d)\epsilon} -\delta_{k\epsilon}+\delta_{(l+d)\epsilon}-\delta_{l\epsilon} \,.
\end{equation}
This can be derived by lifting the generator $\mathcal{L}^{N,L,\epsilon}$ for the evolution of microscopic configurations via the empirical cluster distribution map $C^L$ to the generator for the evolution of empirical measures.  Consider 
$\mathcal{L}^{N,L,\epsilon}$ for a function $H:V^{N,L,\epsilon}\to \bbR$ defined by
% $\cF(V^{N,L,\epsilon})$, the set of functions on $V^{N,L,\epsilon}$,
\begin{equation}\label{eq:def:generator:micro}
\begin{split}
	\mathcal{L}^{N,L,\epsilon}H(\eta)= \frac{1}{L-1}\sum_{x,y=1}^L \sum_{d\in\bN} K(\eta_x,\eta_y,d\epsilon)\epsilon\bra[\big]{H(\eta^{x\stackrel{d\epsilon}{\to} y})-H(\eta)}
\end{split}
\end{equation}
with the state after one transition given by
\begin{equation}\label{eq:def:jump}
	\eta^{x\stackrel{d\epsilon}{\to} y}_z =\begin{cases} \eta_x - d\epsilon \text{ if } z=x,\\
 \eta_y + d\epsilon \text{ if } z= y,\\
 \eta_z \text{ otherwise.}\end{cases}
\end{equation}
For a linear function in the cluster sizes, that is $H(\eta) = \skp[big]{C^L[\eta],h}= G(C^L[\eta])$ for some bounded measurable $h: \bR_+\to \bR$, we have 
\[\mathcal{L}^{N,L,\epsilon}H(\eta)=\sum_{k=1}^{N}\sum_{l=0}^{N-k}\sum_{d=1}^{k} \kappa^L_\epsilon[c](k\epsilon,l\epsilon,d\epsilon) \, \gamma^{\epsilon k,\epsilon l,\epsilon d}\cdot h= \mathcal{Q}^{N,L,\epsilon}G(C^L[\eta])\,, \]
where $\gamma^{\epsilon k,\epsilon l,\epsilon d}\cdot h$ denotes the $\eps$-exchange derivative given in view of~\eqref{eq:def:gamma} as
\[ %\overline{\nabla}^\epsilon_{k,l,d} h = (d\epsilon)^{-1} 
\gamma^{\epsilon k,\epsilon l,\epsilon d}\cdot h= %(d\epsilon)^{-1} \bra[\big]{
 h\bra*{(k-d)\epsilon}-h(k\epsilon)+h\bra*{(l+d)\epsilon}-h(l\epsilon) \,.\]
In this sense, the formulation using $\mathcal{Q}^{N,	L,\epsilon}$ can be considered as a mesoscopic description.
As $N\epsilon/L \to \rho \ge 0$, with $N, L \to \infty$, $\epsilon \to 0$, we show the convergence of the empirical measure process to a weak solution of~\eqref{eq:CGEDG}.
\begin{definition}[Weak solution] \label{def:weak-sol}
	A curve $c\in C([0,T]; \cP(\bR_+))$ is called a weak solution of~\eqref{eq:CGEDG} if for all $f: \bR_+ \to \bR$ Lipschitz, it holds for any $(s,t)\subset \bR_+$ the identity
\begin{equation}\label{eq:weak-sol}
\int_0^\infty f(x) (c_{t}- c_{s})(\d x) 
= \int_{s}^tQ(c_r) \cdot f \dx r 
%\int_0^\infty c_r(\d x) \int_0^\infty c_r(\d y)  \int_0^x  \d z K(x,y,z) \gamma^{x,y,z}\cdot f
\end{equation} 
with the generator $Q$ acting on $f$ by
\begin{equation}\label{eq:def:generator}
Q(c)\cdot f= \int_0^\infty  \int_0^\infty  \int_0^x    K(x,y,z) \, \gamma^{x,y,z} \cdot f  \d z \, c(\d y)c(\d x) ,
\end{equation}
where for $\gamma$ as in~\eqref{eq:def:gamma} the notation $\gamma^{x,y,z}\cdot f = (-f(x)-f(y)+f(x-z)+f(y+z))$ is used.
\end{definition}
Note that any weak solution of \eqref{eq:CGEDG} conserves mass and first moment. Indeed,  $f\equiv1$ and $f(x)=x$ are admissible Lipschitz functions in~\eqref{eq:weak-sol} and for both choices, we get from~\eqref{eq:def:generator} the identity $f\cdot \gamma^{x,y,z} =-f(x)-f(y)+f(x-z)+f(y+z)=0$. Hence, any $c$ solving~\eqref{eq:weak-sol} satisfies $\int c_t(\d x) = \int c_0 (\d x) =1 $ and $\int x \, c_t(\d x)  = \int x\, c_0(\d x)$ for $t\geq 0$.

%\begin{equation}\label{eq:CGEDG}\begin{split}
%\partial_t c (a) = 
%\end{split}\end{equation} 
%
 The proof of convergence is through a compactness argument in the Skorokhod space of the $1$-Wasserstein $W_1$ metric space of probability measures on $\bR_+$ with bounded first moment $\cP^1$ under some growth assumptions on initial data and some growth and regularity assumptions on the kernel $K$. 
\subsection{Assumptions and statement of main results}
 In the following, we  equip the set of probability measures on $\bR_+$ with bounded first moment with the $1$-Wasserstein metric $W_1$. 
 \begin{definition}
The metric space $(\cP^{1}(\bR_+), W_1(\bR_+, |\cdot|_\bR))$ is abbreviated by $\cP^{1}$,
where $W_1(\bR_+)$ denotes the 1-Wasserstein metric, defined in the one-dimensional context by
\[W_1 (\mu,\nu ) = \int_0^\infty| T_{\mu}(x) - T_{\nu}(x)|   \d x ,\] 
with the tail distribution function given as   
$T_{\mu}(x)= \mu([x,+\infty))$. 

The spaces $\cP_{\le \rho}$ and $\cP_{\rho}$ denote the closed subsets of $\cP^1$ with the first moment less than or equal to $\rho$ and with the first moment equal to $\rho$, respectively. 
\end{definition}
\begin{assumption}[Existence]\label{ass:exist}
For all $z\geq 0$, the kernel $K(\cdot,\cdot,z)\in C((0,\infty) \times \bR_{+} ;\bR_+)$ satisfies
\begin{equation*}\label{z-bdd} \tag{K$_1$}
	K(x,y,z) \le C_K (1+x)(1+y) \varphi(z) \,,
\end{equation*}
for some $\varphi: \bbR_+ \to \bbR_+$ such that $z\mapsto \varphi(z)$ and $z \mapsto \varphi(z)z$ are uniformly continuous and integrable.
Moreover, the functions $(K(x,y, \cdot))_{x> 0, y\ge 0}$ are uniformly continuous with a common modulus of continuity.
\end{assumption}
\begin{assumption}[Uniqueness]\label{ass:unique}
For all $z\geq 0$, the kernel $K(\cdot,\cdot,z)\in C^2((0,\infty) \times \bR_{+} ;\bR_+)$ satisfies
\begin{equation*}\label{ass: 1-diff K} \tag{K$_2$}
	|\partial_1 K(x,y,z)|\le (1+y) \varphi(z), \quad |\partial_2 K(x,y,z)|\le (1+x) \varphi(z),
\quad |\partial_1 \partial_2 K(x,y,z)|\le \varphi(z)  
\end{equation*} 
and the boundary condition $K(x,y,x)=0$ for all $x,y\ge 0$.
\end{assumption}
We prove the main result under an initial second moment bound as done in the derivation of the classical EDG system in~\cite{GrosskinskyJatuviriyapornchai2019}.
%The first version of our main theorem is in parallel to \cite{GrosskinskyJatuviriyapornchai2019} where we assume bounded initial second moment. 
We also provide a relaxation on initial data to include a class of functions with slower growth but with additional growth assumption on $\varphi$ which appears in the upper bound of kernel in~\eqref{z-bdd}.
\begin{definition}[Admissible moments]\label{def:generalized:moment}
A function $g\in C^2(\bR_+;\bR_+)$ is an admissible \emph{moment}, denoted by $g\in \cA_0$, provided that $g,g',g'' \ge 0$ on $\bbR_+$ and it satisfies for some $C_1, C_2 \ge 0$ the bounds
\begin{equation*}%\label{eq:cond:generalized:moment}
	\tag{$\calA_0$}
	g'(x+z)\le C_1 \bra*{g'(x) + g'(z)} \quad\text{and}\quad g'(x) \le C_2 \bra*{1+ \frac{g(x)}{1+x}} \qquad\forall x,z\in \bbR_+ \,.
\end{equation*}
In addition, the class of all \emph{strictly superlinear} admissible moments is given by 
\begin{equation*}%\label{eq:def:A}
	\tag{$\calA$}
	\cA = \cA_0 \cap \bigl\{g\in C^2(\bR_+; \bR_+) : \lim_{x\to \infty} g'(x)\to \infty \bigr\}\,. 
\end{equation*}
\end{definition}
\begin{remark}[Examples of admissible moments]
	We are mainly interested in the examples of power functions $x\mapsto x^{1+\alpha}\in \cA$ with exponent $\alpha \in (0,1]$ and entropy type functions $x \mapsto (1+x)\log(1+x)\in \cA$, which are all superlinear admissible moments.
\end{remark}
\begin{remark}[A subclass of admissible moments from literature] We observe that non-negative, non-decreasing and convex functions that satisfy the properties in~\cite[Proposition 2.14]{laurenccot2014} are contained in our class of admissible functions $\cA_0$. In particular, $\cA_0$ contains the class of functions defined in \cite[Definition 2.11]{laurenccot2014}.
\end{remark}
\begin{theorem}[Existence]\label{thm:wllnex} Under Assumption \ref{ass:exist}, assume the initial condition converges, i.e. the pushforward of projection at time zero ${\pi^{-1}_0}_{\#}\bP^{N,L,\epsilon} \to \delta_{c_0}$ weakly for some $c_0 \in \cP_\rho$, $\rho \ge 0$ along a sequence $N\epsilon / L \to \rho$ and $\epsilon \to 0$ sufficiently fast.  Moreover, assume along the same sequence either
\begin{enumerate} 
\item a second  moment bound
\begin{equation}\label{eq:ass-2-moment}
	\sup_{N\epsilon/L \to \rho}\bE^{N,L,\epsilon}[\cM_2(c_0)] <+\infty, \text{ where } \cM_2(c) = \int x^2 c(\d x) \,;
\end{equation}
\item or that for some $f \in \mathcal{A}$, the map $x\mapsto f(x)\varphi(x)$ is uniformly continuous and integrable, and the moment bound
\begin{equation}\label{eq:ass-1+a-moment}
	\sup_{N\epsilon/L \to \rho}\bE^{N,L,\epsilon}[\cM_{f}(c_0)] <+\infty, \text{ where } \cM_f(c) =\int f(x) c(\d x) \,.
\end{equation} 
\end{enumerate} 
Then $\bP^{N,L,\epsilon} \to \bP=\delta_{c}$ converges weakly along a subsequence as probability measures on $D([0,\infty), \cP_{\le \overline{\rho}})$, where $\overline{\rho}=\sup_{N\epsilon/L \to \rho} \frac{N\epsilon}{L}$, and $(c_t)_{t\ge 0} \subset \cP_\rho$ is a weak solution in the sense of Definition~\ref{def:weak-sol}.
%\begin{equation} \label{eq:weak-sol-CGEDG}
%\int_{\bR_+} f(x) (c_{t } - c_{s})(\d x) = \int_{s}^{t} \d r   \int_0^\infty c_r(\d x) \int_0^\infty c_r(\d y)   \int_0^x \d z K(x,y,z)  f\cdot \gamma^{x,y,z} \end{equation} 
%for als functions $f:\bR_+\to \bR$ Lipschitz and for each $t \ge s \ge 0 $.
\end{theorem}

 In addition, we provide sufficient conditions along the initial data of the finite sequence to have Lebesgue absolutely continuous measure in the limit. 
\begin{proposition}\label{prop:abscty of soln}
	Suppose the law  $\bP^{N,L,\epsilon}=C^L_{\#}\mu_{N,L,\epsilon}$ is induced by a microscopic evolution where $\mu_{N,L,\epsilon}(t,\eta) \in \cP(V^{N,L,\epsilon})$ is the law of a Markov process with the microscopic generator~$\cL^{N,L,\epsilon}$ and $\bP^{N,L,\epsilon}\to \bP$ weakly in  $D([0,\infty), \cP_{\le\overline{\rho}})$. Moreover assume for some $b>0$
	\begin{equation}\label{eq:entropy-bdd}
		\sup_{\frac{N\epsilon}{L} \to \rho}\frac1L H(\mu_{N,L,\epsilon}(0) | \nu_{L,\epsilon,b} )<+\infty \,,
	\end{equation}
	where $\nu_{L,\epsilon,b} \in \cP(V^{N,L,\epsilon})$,  $V^{N,L,\epsilon}=\{ \eta \in (\bN_0 \epsilon)^{L}   \}$, is the reference measure defined for any $A\subseteq V^{N,L,\eps}$ by $\nu_{L,\epsilon,b}(A)
	% =\sum_{\eta \in A} \Phi_{L,\epsilon,b}(\eta)
	= \sum_{\eta\in A}\prod_{i=1}^L g_{\epsilon,b}(\eta_i)$ 
	with
	\begin{equation}\label{eq:def:g:measure}
		g_{\epsilon,b}(m\epsilon) =b \int_{m\epsilon}^{(m+1)\epsilon} e^{-b x} \d x 
		\quad \text{ for } m\in \bN_0,
		\qquad 
		\text{such that}
		\qquad
		g_{\eps,b} \in \cP(\eps \bN_0) \,.
	%	b^{-1} =\sum_{i\in\bN_0} \int_{i \epsilon}^{(i+1)\epsilon} e^{-b x} \d x.
	\end{equation}
	Then for $\bP$ is concentrated on paths $(c_t)_{t\ge0}$ such that the probability measure $c_t$ on $\bR_+$ has a Lebesgue density for almost all $t\geq 0$. 
\end{proposition}
\begin{remark}
	\begin{enumerate}
		\item This proposition gives the condition on initial conditions along the finite particle system so that the limit measure will be absolutely continuous with respect to Lebesgue. This is only \emph{a} way to construct an absolutely continuous limit measure. 
		\item A simple calculation shows if there exists a $b>0$ such that \eqref{eq:entropy-bdd} holds, then it also holds for every $b>0$. 
		\item We provide in Lemma~\ref{lemma:density} a sufficient conditions ensuring the existence of an initial measure $\mu_{N,L,\eps}(0)$ satisfying~\eqref{eq:entropy-bdd}. In particular, we are able to provide such examples, so that  $ C_{\#}^L \mu_{N,L,\epsilon}(0) \to \delta_{c_0}$  for initial measures $c_0\in \calP(\bR_+)$ having a Lebesgue density that is uniformly continuous and uniformly bounded (see Corollary~\ref{cor:densitfy:sufficient}). Proposition \ref{prop:abscty of soln} then implies that for such an initial condition, the weak solution in the sense of Definition~\ref{def:weak-sol} has Lebesgue density for almost all $t\ge 0$.
 
	\end{enumerate}
\end{remark}
Therefore, under the additional assumptions on the relative entropy of the initial data, the limit point is a delta measure on the weak solution $c$ with Lebesgue density.
We conclude this work with the uniqueness of solutions to~\eqref{eq:CGEDG} under more restrictive conditions on the kernel~$K$.
\begin{theorem}[Uniqueness]\label{thm:unique}
	Under Assumption \ref{ass:exist} and Assumption~\ref{ass:unique}, the weak solutions to~\eqref{eq:CGEDG} in the sense of Definition~\ref{def:weak-sol} having a Lebesgue density are unique for a given initial condition $c_0\in \cP_{\rho}$ with $\rho\geq 0$.
	%and with the same initial conditions, we have $c_t=d_t$ for all $t>0$.
\end{theorem} 
  \begin{remark}
Together with Proposition \ref{prop:abscty of soln} and Theorem~\ref{thm:unique}, we have the weak law of large numbers for the particle process under Assumptions \ref{ass:exist}, \ref{ass:unique} and suitable conditions on the initial data.  
 \end{remark}

\section{Convergence of the particle system: weak law of large numbers}%\label{sec:convergence}

The proof of Theorem \ref{thm:wllnex} is done in several steps. 
In  Proposition \ref{prop:tightness in dexg}, we use a tightness argument via the Aldous tightness criterion~\cite{aldous1978stopping} for the Skorokhod space $D([0,\infty), \cP_{\le\overline{\rho}})$ to extract a limit measure along a subsequence. 
To verify the Aldous tightness criterion, we need to characterize relative compactness in $\cP^1$ in Lemma \ref{lem: compactness in dexg} and to show a uniform in time estimate for certain moments in Proposition~\ref{prop: propa of 2. moment}.
Then we identify the limit measure as the delta measure at the solution of~\eqref{eq:weak-sol} in  Proposition~\ref{prop:ident-limit}.

\begin{proposition}[Tightness in $D([0,\infty), \cP_{\le\overline{\rho}})$]\label{prop:tightness in dexg}
Under Assumption  \ref{ass:exist}, assume along a sequence $N\epsilon / L \to \rho$ and $\epsilon \to 0$ sufficiently fast either 
\begin{enumerate}[(i)]
\item the uniform second moment bound~\eqref{eq:ass-2-moment};
%\begin{equation}
%	\label{eq:ass-second-moment}\sup_{N\epsilon/L \to \rho}\bE^{N,L,\epsilon}[\cM_2(c_0)] <+\infty \,; \end{equation} 
\item or there exists some $f \in \cA$, such that $x\mapsto f(x)\varphi(x)$ is uniformly continuous and integrable and the uniform moment bound~\eqref{eq:ass-1+a-moment}.
%\begin{equation}\label{eq:ass-1+a-moment:prop}
%	\sup_{N\epsilon/L \to \rho}\bE^{N,L,\epsilon}[\cM_{f}(c_0)] <+\infty.
%\end{equation}
%along a sequence $N\epsilon / L \to \rho$ and  $\epsilon \to 0$ sufficiently fast.
\end{enumerate} 
Then the family of measures  $(\bP^{N,L,\epsilon})$  is tight on the Skorokhod space $D([0,\infty), \cP_{\le\overline{\rho}})$.
\end{proposition}
\begin{remark} In the first statement, we only assume $x\mapsto \varphi$ and $x\mapsto x\varphi(x)$ are uniformly continuous and integrable but no assumptions on $x\mapsto x^2 \varphi(x)$.
\end{remark}

Before coming to the proof of Proposition~\ref{prop:tightness in dexg}, we characterize the relative compactness in $\calP^1$.
\begin{lemma}[Relative compactness in $\cP^1$]\label{lem: compactness in dexg}
A sequence of probability measures $(\mu_n)_n$ is  relative compact in $\cP^1$ if and only if there exists  a function $\phi:\bR_+ \to \bR$ monotonically increasing with  $\lim_{x\to +\infty}\phi(x) =  +\infty$  such that the tail distribution functions $(T_{\mu_n})_n$ of $(\mu_n)_n$ satisfies 
$\sup_{n\in \bN} \int_0^\infty \phi(x) T_{\mu_n}(x) \d x < +\infty.$
\end{lemma}
\begin{proof}%[Proof of Lemma \ref{lem: compactness in dexg} ]
Since the convergence with respect to $W_1$ is equivalent to the $L^1$ convergence of the tail distribution function, the compactness in $W_1$ is equivalent to the compactness of $(T_n)_n$  in $L^1$. For the reverse direction, we need to identify the limit which is right continuous and $\lim_{x\to \infty}T(x)=0 = 1- \lim_{x\to 0^+}T(x)$.

Since $T_n$ are monotone functions, by Helly's selection theorem, the strong compactness in $L^1$ is equivalent to weak compactness in $L^1$, see e.g.~\cite{Diestel-Ruess-Schachermayer-1993}. By Dunford-Pettis theorem~\cite{laurenccot2014}, it is equivalent to $(T_n)_n$ being uniformly bounded in $L^1$, uniformly integrable 
\[\lim_{\epsilon \to 0}\sup_{A: |A| \le \epsilon} \sup_n \int_{A} T_{n}(x) \d x = 0 \] and tight 
\[\limsup_{k\to\infty} \sup_{n} \int_k^\infty T_n(x) \d x = 0.\]
Since $0 \le T_n\le 1$, the uniform integrability is automatic and tightness implies uniformly bounded in $L^1$.

Now we use an argument by \cite[Proposition 1]{canizo2006lemma} to  show the equivalence of tightness 
\[ \limsup_{k\to\infty} \sup_{n} \int_k^\infty T_n(x) \d x = 0\] 
to the existence of  a function $\phi:\bbR_+\to \bbR_+$ with $\phi(x) \to \infty$ as $x\to\infty$ and such that 
\[\sup_n \int_0^\infty \phi(x) T_n(x) \d x <+\infty.\]
Suppose such a $\phi$ exists, then
\[\sup_n \int_k^\infty T_n(x) \d x \le \phi(k)^{-1} \sup_n \int_0^\infty \phi(x) T_n(x) \d x \to 0 \qquad\text{ as } k\to \infty \, \]
Conversely, let $F(y)=\sup_n \int_y^\infty T_n(x) \d x$, then $\lim_{y \to 0^+} F(y) = 0$. 
We define the sequences
\[ a_n= \inf\set*{x>0: F(x)< \frac{1}{n^2}},\quad x_0=0, \quad x_{n+1}=\max\bigl\{x_n+1, a_{n+1} +1 \bigr\} . \] 
Then $x_n \to \infty$ and $F(x_n)\le \frac{1}{n^2}$. Let $\phi(x)=\sum_{m=0}^\infty \1_{[x_m,\infty)}(x)=n+1$ for $x\in[x_n,x_{n+1}]$ and any $n\in \bN_0$. 
Then, we have the claim, since
\[\int \phi(x) T_n(x) \d x = \sum_{m=0}^\infty \int_0^\infty \1_{[x_m,\infty)}(x) T_n(x) \d x\le \sum_{m=0}^\infty F(x_m)<+\infty. \qedhere \] 
\end{proof}
Instead of using relative compactness in the $1$-Wasserstein metric, for technical reasons, we control the propagation of the moments in time. The following lemma allows us to use the $f\in \mathcal{A}_0$ for an admissible moment (Definition~\ref{def:generalized:moment}) to obtain relative compactness in the $1$-Wasserstein metric.
It will be used in the proof of Proposition \ref{prop:tightness in dexg}.
\begin{lemma}[Relation between the  $f$  moment and relative compactness in $\cP^1$] \label{lem:compactness via second moment}For any $N,L, \epsilon$, any $c \in \hat{V}^{N,L,\epsilon}$ and any function $f\in \cA_0$ exist a constant $C_f$ such that
\[\int_0^\infty f'(x) T_{c}(x) \d x\le C_f \bra*{\frac{N\epsilon}{L} +\cM_f(c)} + O(\epsilon).\] 
In particular, if there exists $f \in \cA$ such that  $\sup_{c\in \mathcal{S}}\cM_f(c)<+\infty$ for $\mathcal{S}\subset \bigcup_{\frac{N\epsilon}{L}\to \rho} \hat{V}^{N,L,\epsilon}$, then the set $\mathcal{S}$ is relative compact in $\cP^1$.
\end{lemma}
%\begin{remark}
%In this statement, instead of the second moment, one could use $1+\alpha \in (1,2]$ moment.  Moreover, we will see in the Remark \ref{rmk:more-general-moment}, there is a trade off between the integrability of $\varphi$, then  integrability of the initial data. This constraint comes from showing an appropriate Gronwall estimate.
%\end{remark}
\begin{proof}
Since $c\in V^{N,L,\epsilon}$, we can write its tail distribution as
\[\begin{split}\int_0^\infty f'(x) T_c(x) \d x &= \sum_{i=0}^{N}\int_{i \epsilon}^{(i+1)\epsilon} f'(x)  T_{c}(i \epsilon) \d x
= \sum_{i=0}^{N} T_{c}(i\epsilon)  \bigl(f((i+1) \epsilon) -f(i\epsilon)\bigr) \\
&=  \sum_{j=0}^N  c(j\epsilon)   (f((j+1)\epsilon) - f(0))  \\
&\le \sum_{j=0}^N  c(j\epsilon)   f'((j+1)\epsilon)  (j+1)\epsilon \\
& \le C_1 \sum_{j=0}^N  c(j\epsilon)  (f'(j\epsilon) (j+1)\epsilon + f'(\epsilon) (j+1)\epsilon) \\
& \le   C_1 \sum_{j=0}^N  c(j\epsilon) \bigg(C_2 (1+\frac{f(j\epsilon)}{1+j\epsilon}) (j+1)\epsilon + f'(\epsilon) (j+1)\epsilon)\bigg)\\
&\le C_1 \sum_{j=0}^N c(j\epsilon) \bigg(\big(C_2 + f'(\epsilon)\big)(j+1)\epsilon  + C_2 f(j\epsilon)\bigg) \\
&\le C_f \sum_{i = 0}^N \bigl(i\epsilon + f(i\epsilon)\bigr) c(i \epsilon) + O(\epsilon)\\
&=   C_f\bra*{\frac{N\epsilon}{L} +\cM_f(c)} + O(\epsilon)
\end{split}\]
where we  choose $C_{f}=C_1(C_2 + 	f'(1))$.
%\[\begin{split}\int_0^\infty x^{\alpha} T_c(x) \d x &= \sum_{i=0}^{N}\int_{i \epsilon}^{(i+1)\epsilon} x^{\alpha} T_{c}(i \epsilon) \d x
%\le \sum_{i=0}^{N} T_{c}(i\epsilon)  (i+1)^{\alpha} \epsilon^{\alpha+1} \\
%&=  \epsilon^{\alpha+1} \sum_{j=0}^N  c(j\epsilon)   \sum_{i=1}^{j+1} i^\alpha  \\
%&\le  \epsilon^{\alpha+1} \sum_{j=0}^N  c(j\epsilon)\int_{1}^{j+2} x^{\alpha} \d x \\
%& \le \epsilon^{\alpha+1}  \frac{1}{1+\alpha} (c(0)2^{\alpha+1} + \sum_{j=1}^N  c(j\epsilon) j^{\alpha+ 1}\big(1+\frac{2}{j}\big)^{\alpha+1})\\
%& \le  \frac{1}{1+\alpha} 3^{\alpha+1} (\sum_{j=0}^N c(j\epsilon) (j\epsilon)^{\alpha+1} + \epsilon^{\alpha+1})
%\end{split}\]

%\[\begin{split}\int_0^\infty x^{\alpha} T_c(x) \d x &= \sum_{i=0}^{N}\int_{i \epsilon}^{(i+1)\epsilon} x^{\alpha} T_{c}(i \epsilon) \d x
%=\sum_{i=0}^{N} T_{c}(i\epsilon)  (i+1)^{\alpha} \epsilon^{\alpha+1} \\
%&=   \epsilon^2 \sum_{i=0}^N \frac{i^2 + 3 i +1 }{2} c(i\epsilon) \le\epsilon^2 \sum_{i=0}^N (i^2 + \frac32 )  c(i\epsilon)=\sum_{i=0}^N (i\epsilon)^2 c(i\epsilon) +\frac32 \epsilon^2.
%\end{split}\]
Finally, with  Lemma \ref{lem: compactness in dexg}, we conclude the claim for relative compactness of the set $\mathcal{S} \subset \cP^1$.
\end{proof}
The next lemma gives an approximation result about the precise rate for $\epsilon \to 0$ to ensure the convergence of discrete sums to integrals in terms of the uniform modulus of continuity of the integrand. We apply the lemma to $\varphi$, which is the reason to assume its uniform continuity.
\begin{lemma}\label{lem:uniform-bdd-phi} 
	Given a uniform continuous function $h:\bbR_+\to \bbR$ with $\omega_h:\bbR_+ \to \bbR_+$ the modulus of uniform continuity.
	Fix $\rho>0$ and suppose the sequence $(N,L,\epsilon)$ satisfies $\frac{N\epsilon}{L}\to \rho$ and $L \omega_h(\epsilon) \to 0$ as $N,L\to \infty$ and $\eps\to 0$, then 
	\[
	\lim_{N\epsilon/L \to \rho} \epsilon \sum_{d=1}^N h(d \epsilon)  = \int_0^\infty h(x) \d x \,.
	\]
\end{lemma}
\begin{proof}
We consider the difference
\[\abs*{\epsilon \sum_{d=1}^N h(d \epsilon) - \int^{N\epsilon}_0 h(x) \d x }
= \abs*{ \sum_{d=1}^N \int_{(d-1)\epsilon}^{ d\epsilon} ( h(d \epsilon) - h(x) )\d x  }
\le  \frac{N \epsilon}{L}  L\omega_h (\epsilon)   \to 0 \text{ as }  N\epsilon/L \to \rho. \]
Hence, we get the conclusion
\[ \lim_{N\epsilon/L \to \rho} \epsilon \sum_{d=1}^N h(d \epsilon) = \lim_{N\epsilon/L \to \rho} \int_0^{N\epsilon} h(x) \d x = \int_0^\infty h(x) \d x. \qedhere
\]
\end{proof}
It is due to the following lemma that we consider the class of functions $\cA_0$.
\begin{lemma}\label{lem:control of ex-grad} Let $f:\bR_+ \to \bR_+$ be increasing such that there exists $g:\bR_+ \to \bR_+$ with 
\begin{equation} \label{cond:wk-ex-grad} 
	\bigl\lvert -f(x)-f(y)+f(x-z) + f(y+z) \bigr\rvert 
	\le g(z)\bigg( \Big(1 + \frac{f(y)}{1+y}\Big)+   \Big(1 + \frac{f(x)}{1+x}\Big)\bigg)
\end{equation} 
for $x,y,z \ge 0$, such that $x\mapsto g(x) \varphi(x)$ is uniformly continuous and integrable on $\bR_+$. 
Then the following statements hold.
\begin{enumerate}[(i)] 
\item There exists $C_{g\varphi}$ such that for any $c\in\hat{V}^{N,L,\epsilon}$ it holds
\[\bigg|\mathcal{Q}^{N,L,\epsilon}\int f(x) \d c(x)\bigg| \le C_{g\varphi}\Bigl(1+\frac{N}{L}\Bigr)\Bigl(1+\frac{N}{L}+ \int f \d c\Bigr).\]

\item The function $f(x)=x^2$ satisfies  the  condition \eqref{cond:wk-ex-grad}  with $g(z) = z$.

\item %\label{cd: exchange gradient sufficient condition} 
Any function $f\in \cA_0$ satisfies the sufficient condition $f(x+z)-f(x) \le g(z) \big(1 + \frac{f(x)}{1+x}\big) $ and $ f(x)-f(x-z)\le g(z)  \big(1 + \frac{f(x)}{1+x}\big)$ with the choice $g(z) \propto 1+z + f(z)$.\\
In particular, the function $f(x)=x^{1+\alpha}$ for $\alpha \in (0,1)$ satisfies~\eqref{cond:wk-ex-grad} with $C_1 =1$, $C_2=1+\alpha$, and the function $f(x)=(1+x)\log(1+x)$ with $C_1=C_2=1$.
\end{enumerate}
 \end{lemma}
\begin{proof}
\begin{enumerate}[(i)] 
	\item 
Define $h(x)=1+\frac{f(x)}{1+x}$, the assumption implies the estimate
\[
\begin{split}\MoveEqLeft
	\abs*{ \mathcal{Q}^{N,L,\epsilon}\int f(x) \d c(x)} = \abs*{ \sum_{k=1}^N\sum_{l=0}^{N-k}\sum_{d=1}^k \kappa_\epsilon^L[c](k\epsilon, l\epsilon,d\epsilon) \sum_{i=0}^N \gamma^{k\epsilon,l\epsilon,d\epsilon}_{i \epsilon} f(i\epsilon)} \\
&=\abs*{ \sum_{k=1}^N\sum_{l=0}^{N-k}\sum_{d=1}^k \kappa_\epsilon^L[c](k\epsilon, l\epsilon,d\epsilon)  \bigl(-f(k\epsilon)-f(l\epsilon)+f((k-d)\epsilon) + f((l+d)\epsilon)\bigr) } \\
&\le 2 \epsilon \sum_{k=1}^N (1+k\epsilon)c(k\epsilon) \sum_{l=0}^{N-k} c(l\epsilon) (1+l\epsilon)   \sum_{d=1}^k  \varphi( d \epsilon)  g(d \epsilon) (h(k\epsilon)+h(l\epsilon)).
\end{split}\]
Using $(1+x)h(x) = 1+x+f(x)$ and the uniform continuity of  $x\mapsto g(x) \varphi(x)$, we have 
\[\begin{split} 
	\abs*{ \mathcal{Q}^{N,L,\epsilon}\int f(x) \d c(x)} & \le  4    \bra*{ 1+\frac{N}{L}} \sum_{k=1}^N (1+k\epsilon+ f(k\epsilon)) c(k\epsilon)\, \epsilon \sum_{d=1}^k \varphi(d \epsilon)  g(d \epsilon)\\
& \le 4  C_{\varphi g}  \bra*{1+\frac{N}{L}} \bra*{ 1+ \frac{N}{L}+ \int f \d c } .
\end{split} \] This is due to $\lim_{N\epsilon/L \to \rho} \epsilon \sum_{d=1}^k \varphi(d \epsilon)  g(d \epsilon) = \int \varphi(x) g(x) \d x$ by Lemma \ref{lem:uniform-bdd-phi} so it is bounded.
\item
For $f(x)=x^2$, we have by explicit calculation 
$- 2z x \le -x^2-y^2+(x-z)^2+(y+z)^2 = 2z(-x+y+z) \le 2 z y$ so the condition  \eqref{cond:wk-ex-grad} is satisfied.

%For $f(x)=x^{1+\alpha}$,$\alpha\in (0,1]$,
%we have \[f(x+z) - f(x) = f'(x) z + \frac12 f''(x) z^2 + \int_0^z f^{(3)}(x+t)(z-t)^2 \d t \le C(f'(x)+f''(x))(z+z^2) \]
%\[ f(x) - f(x-z)=  f'(x)z - \frac12 f''(x) z^2 +\int_0^z f^{(3)}(x-t) (z-t)^2 \d t \le f'(x)z  \]
%Note that $f^{(3)}(x) = (\alpha+1)\alpha(\alpha-1) x^{\alpha -2} \le 0$ for $x\ge 0$, $\alpha \in (0,1]$, $f^{(1)}(x)= (1+\alpha) x^{\alpha}\ge 0, f^{(2)}(x) = (\alpha+1)\alpha x^{\alpha -1} \ge 0$. However as $f'' \to \infty$ at $0$, we use the first bound only for $x \ge 1$, 
%\[f(x+z) - f(x) \le C(f'(x)+f''(1))(z+z^2).\]
%
%For $x<1$, we can simply use \[f(x+z)-f(x) \le f(x+z) \le f(z+1) = (z+1)^{1+\alpha}\le (z+1)^2.\] As 
%$(1+f'(x))\le C_\alpha (1+ x^\alpha) \le C_\alpha (1+\frac{x^{1+\alpha}}{1+x})$. The sufficient condition is satisfied.
\item 
More generally, let $f$ be such that $f'' \ge 0$, $f' \ge 0$, $f'(x+z)\le C_1 (f'(x) + f'(z))$ and $f'(x) \le C_2 \bigl(1+ \frac{f(x)}{1+x}\bigr)$.
The convexity of $f$ implies
\[f(x)-f(x-z)\le f'(x) z  \le \bra*{1+ \frac{f(x)}{1+x}}  C_2 z. \]
Then by convexity and sublinearity of $f'$,  we get
\[f(x+z)-f(x) \le f'(x+z) z \le C_1 (f'(x) +f'(z)) z.\]
By monotonicity, there exists $x_0$ such that $f'(x_0) > 0$ and for all $x \le x_0:$ $f'(x) \le f'(x_0)$. Hence, we get for $x \ge x_0$ the estimate
\[ \begin{split}\bra*{f'(x) +f'(z)} z &= f'(x) \bra*{1 + \frac{f'(z)}{f'(x)}} z \le   f'(x) \bra*{1 + \frac{f'(z)}{f'(x_0)}} z\\
 &\le C_2 \bra*{1+\frac{f(x)}{1+x}} \bra*{1 +\frac{C_2}{f'(x_0)} \bra*{1+\frac{f(z)}{1+z}}} z\\
 &\le  C_2 \bra*{1+\frac{f(x)}{1+x}} \bra*{1 +\frac{C_2}{f'(x_0)} (z+f(z)}.
\end{split}\]
For $x \leq x_0$, 
\[\bra*{f'(x)+f'(z)}z \le \bra*{f'(x_0)+f'(z)} z \le \bra*{f'(x_0) +C_2 (1+\frac{f(z)}{1+z}) } z \le \bra*{f'(x_0) + C_2} \bra*{z + f(z)}. \]
Combining the   estimates above, we can choose for $f'(x_0)\ge C_2$, $\tilde C_1\ge \max(1/2, C_1)$
\[\max\bra*{f(x) - f(x-z) ,  f(x+z) - f(x) }\le \bra*{1+\frac{f(x)}{1+x}}   2 \tilde{C_1} f'(x_0) \bra*{1+ z +f(z)} .\]
It is a simple calculation to verify  the functions $x\mapsto x^{1+\alpha}$ for $\alpha \in (0,1)$ and  $x\mapsto (1+x)\log(1+x)$ are in $\cA_0$. \qedhere
\end{enumerate}
\end{proof}
With the characterization of relative compactness in $\cP^1$ and the technical estimates for functions $f\in \calA$, we next show the propagation of the moment bounds assumed in Theorem~\ref{thm:wllnex}.
\begin{proposition}[Propagation of bounded moments]\label{prop: propa of 2. moment}
	Under Assumption \ref{ass:exist}, along a sequence $\frac{N\epsilon}{L} \to \rho$ with $\epsilon \to 0$ sufficiently fast with initial datum satisfying either~\eqref{eq:ass-2-moment} or~\eqref{eq:ass-1+a-moment}, then the process satisfies for some $C>0$ the moment bound
	\[
	\bE^{N,L,\epsilon}[\cM_2(c_t)] \le e^{C t} (\bE^{N,L,\epsilon}[\cM_2(c_0)]+Ct) \,.
	\]
	%where $C$ depends on $\varphi, \rho$ and  $K$.
	%If, in addition for a function $f \in \cA_0$, the map $x\mapsto f(x)\varphi(x)$ is uniformly continuous and integrable, then there exists $C>0$ such that
	or
	\[
	\bE^{N,L,\epsilon}[\cM_{f}(c(t))] \le e^{C t} \bigl(\bE^{N,L,\epsilon}[\cM_{f}(c_0)]+Ct \bigr) \,,
	\]
	respectively.
\end{proposition}
\begin{proof}%[Proof of Proposition \ref{prop: propa of 2. moment}]
For the proof, we consider an arbitrary $f\in \{x\mapsto x^2 \}\cup \cA_0$.
By the property of the infinitesimal generator
\[\frac{\d}{\d t} \bE^{N,L,\epsilon}[\cM_{f}(c_t)] =\bE^{N,L,\epsilon}[\mathcal{Q}^{N,L,\epsilon} \cM_{f}(c_t)].\]
Under the appropriate assumptions on $\varphi$ in the case if $f(x)=x^2$ or $f\in\cA_0$, we can apply Lemma \ref{lem:control of ex-grad} to obtain the bound
\[ \mathcal{Q}^{N,L,\epsilon} \cM_{f}(c) \le C_{\varphi} \bra*{1+\frac{N}{L}}\bra*{1+\frac{N}{L}+\cM_{f}(c)} \,.
\] 
Hence, we can close the differential inequality for some $C>0$ to
\[\frac{\d}{\d t} \bE^{N,L,\epsilon}[\cM_{f}(c_t)]\le C \bE^{N,L,\epsilon}\pra*{(1+\cM_{f}(c_t))} \]
and conclude by Gronwall's lemma
\[
	\bE^{N,L,\epsilon}[\cM_{f}(c_t)] \le e^{C t} \bigl(\bE^{N,L,\epsilon}[\cM_{f}(c_0)]+Ct \bigr) \,. \qedhere
\]
\end{proof}

\begin{proof}[Proof of Proposition \ref{prop:tightness in dexg}]
%First we state the compactness criteria on Skorokhod space $D([0,\infty), V)$, for complete metric space $(V,d)$. A set $A\subset D([0,m], V)$ is precompact if and only if 
%\begin{enumerate}[i)]
%\item there is a compact set $K \subset V$ such that $c(t) \in K$ for all $t \in [0,m]$ and $c \in A$.
%\item $\lim_{\delta \to 0} \sup_{c \in A} w'_m(c, \delta)=0$.
%\end{enumerate}
On a Skorokhod space $D([0,\infty), V)$ over a complete metric space~$(V,d)$, a family of probability measures $(P_n)_n$ on $D([0,\infty), V)$  is tight \cite[Chapter 3, Theorem 7.2, Corollary 7.4]{ethier2009markov}, \cite[Theorem 16.8]{billingsley2013convergence} if and only if 
\begin{enumerate}[i)]
\item\label{enum:tight1}  for each $t\in [0,\infty)$ and for each $\eta >0$, there exists a compact set $K_{\eta,t} \subset V$, $n_0$ such that  
\[
	P_n(c: c(t)\not \in K_{\eta,t} ) \le \eta \qquad \forall n \ge n_0 \,,
\]
\item %\label{enum:tight2} 
for each $m,\epsilon>0$ 
\[
	\lim_{\delta \to 0}\limsup_{n\to \infty} P_n(c: w_m'(c,\delta)\ge \epsilon) = 0 \,,
\] 
where $w'_m(c,\delta)= \inf_{\{t_i\}}\max_{i}\sup_{s,t \in [t_{i-1},t_i]} d(c(s),c(t))$ where the infimum ranges over all partitions of $\{0=t_0<t_1<...<t_{n-1}<T\le t_n\}$ with $\min_{1\le i\le n} (t_i -t_{i-1})>\delta$ and $n \ge 1$, $w_m(c,\delta)=\sup_{t\in [0,m]}\sup_{|t-s|<\delta} d(c(t),c(s)) $.
\end{enumerate} 
Now for the space $(V,d)= (\cP^1,W_1)$ and for  the  sequence  $(\bP^{N,L,\epsilon})_{N\epsilon/L\to \rho}$, if we for some $\phi:\bR_+\to \bR$ with $\phi \nearrow \infty$ and each $t\geq 0$ the limit
\[\lim_{G\to \infty} \limsup_{N\epsilon/L \to \rho} \bP^{N,L,\epsilon} \bigg(c:  \int_0^\infty \phi(x) T_{c_t}(x) \d x \ge G \bigg) = 0 \,, \]
then $(\bP^{N,L,\epsilon})$ satisfies \ref{enum:tight1} because by Lemma \ref{lem: compactness in dexg},  for each $G>0$, the set $\bigl\{c(t) :\int_0^\infty \phi(x) T_{c_t}(x) \d x < G\bigr\}$ is relatively compact in $(\cP^1,W_1)$ for any given $t\geq 0$. By Lemma \ref{lem:compactness via second moment}, it suffices to show for each $t\in [0,\infty)$ the limit
\begin{equation}\label{eq:pointwise-in-time compact}
	\lim_{G\to \infty} \limsup_{N\epsilon/L\to \rho} \bP^{N,L,\epsilon} \bigg(c:  \cM_f(c_t) \ge G \bigg) = 0 \,,
\end{equation} 
where $f \in \{x\mapsto x^2\} \cup \cA$.
To show \eqref{eq:pointwise-in-time compact},  for $t\in [0,\infty)$, we use Markov's inequality to get
\[ \bP^{N,L,\epsilon} \Bigl(c:    \cM_f(c_t)
 \ge G \Bigr) \le \frac{1}{G} \bE^{N,L,\epsilon}\Bigl[ \cM_f(c_t) \Bigr]=\frac1G \bE^{N,L,\epsilon}\bigl[\cM_f(c_t)\bigr].\] 
 Under the assumption of uniform $f$ moment bound for the initial data, we get
\[
\sup_{N\epsilon/L \to \rho}\bE^{N,L,\epsilon}\bigl[\cM_f(c_0)\bigr] <+\infty . 
\] 
So, we conclude with Proposition~\ref{prop: propa of 2. moment} and the assumptions~\eqref{eq:ass-2-moment} or \eqref{eq:ass-1+a-moment} for each $t\in [0,\infty)$ for some $C_t$ independent of $N,L,\epsilon$ and any $G>0$ the bound
\[ 
	\bP^{N,L,\epsilon}\Bigl(c:    \cM_f(c_t)
	\ge G \Bigr)  \le \frac{1}{G}C_t \,.
\]
This implies \eqref{eq:pointwise-in-time compact} and thus \ref{enum:tight1}.

To verify ii), we observe by \cite[Theorem 16.9]{billingsley2013convergence}, \cite{aldous1978stopping}, \cite[Theorem 2.2.2]{joffe1986weak}, that it is left to verify Aldous' compactness criterion criterion:
For all $\mu, \eta, m>0$, there exists $\delta_0$ and $N_0$ such that if $\delta\le \delta_0$ and $N\ge N_0$ in the sequence $(\frac{N\epsilon}{L})_{N \in \bN}$ and for any stopping times $\tau$ adapted to the random element  $c$  with $\tau \le m$ it holds
%For each $m$, if $\delta < m/2$, then there exists a partition with $\delta <t_i - t_{i-1} \le 2\delta$ so $w'_m(c,\delta)\le w_m(c , 2 \delta)=\sup_{s,t \in [0,m]: |s- t|\le 2\delta} \d_{\Exg}(c(s),c(t))$. 
%Hence $ P_n(c: w_m'(c,\delta)\ge \epsilon) \le  P_n(c: w_m(c,2\delta)\ge \epsilon)$. Our goal is then to show for each $\eta > 0$
\begin{equation}\label{eq:AldousCriterion}
	\bP^{N,L,\epsilon}(c:W_1(c_{\tau+\delta}, c_{\tau})\ge \mu) \le \eta. 
\end{equation}
We will obtain the bound~\eqref{eq:AldousCriterion} by Markov's inequality and therefore consider 
\begin{equation}\label{eq:expect_d_ex(c,c+)}
	\bE^{N,L,\epsilon}\bigl[W_1(c_{\tau+\delta}, c_{\tau})\bigr] = \bE^{N,L,\epsilon}\bigl[\| T_{c_{\tau+\delta}}- T_{c_{\tau}}\|_{L^1(\bR_+)}\bigr]= \bE^{N,L,\epsilon}\biggl[\sum_{i=0}^N \epsilon \biggl|\sum_{j= i}^N (c_{\tau+\delta}(j \epsilon) - c_{\tau}(j \epsilon))\biggr|\biggr].
\end{equation}
By the martingale problem formulation of Markov processes \cite[Chapter 4, Proposition 1.7, Chapter 1 Proposition 1.5]{ethier2009markov}, \cite[Section 6.3]{eberle2023markov}, for any given $N,L$ and $\epsilon$, the generator is bounded.
In addition, we have a vector-valued equation on $(i\epsilon)_{i=0}^N$ given by
\begin{equation}\label{eq:martingale problem finite}
	c_{t+\delta}- c_{t} = \int^{t+\delta}_{t} \mathcal{Q}^{N,L,\epsilon}c_r \dx r+ (M_{t+\delta}^L- M_t^L),\qquad \forall t \ge 0 \,,
\end{equation}
where $(M^L_t)_{t\ge0}$ is a martingale under $\bP^{N,L,\epsilon}$. We omitted the superscript $L$ of $c$ in the equation above but the law of $c$ is given by $\bP^{N,L,\epsilon}$. More precisely, the above equation should be understood as the expression for $f(c_{t+\delta})- f(c_t)$ with $f(c) = \operatorname{Proj}_{i\epsilon}(c) = c_{i\epsilon} $ for each $i \in \{0,..,N\}$ and $\operatorname{Proj}_{x}: \bR_{+}^{\bR_{+}} \to \bR$ defined by $c \mapsto c(x)$. Recalling the definition~\eqref{eq:def:gamma} of $\gamma$, we deduce from the bound
\begin{equation}\label{eq:gamma:bound}
	\sum_{i=0}^{N}\biggl\lvert \sum_{j\ge i} \gamma_{j\epsilon}^{k\epsilon, l\epsilon, d\epsilon}\biggr\rvert= \sum_{i=0}^N \bigl\lvert -\1_{(k-d,k]}(i) + \1_{(l,l+d]}(i)\bigr\rvert \le 2 d,
\end{equation}
an estimate of the integral in \eqref{eq:expect_d_ex(c,c+)} using \eqref{eq:martingale problem finite}. For $\tau$ a bounded stopping time, we have
\[\begin{split}
\MoveEqLeft	\sum_{i=0}^N \epsilon \bigg|\sum_{j= i}^N \int_\tau^{\tau+\delta} \mathcal{Q}^{N,L,\epsilon}c_r(j\epsilon) \d r  \bigg| 
\le \sum_{k=1}^N \sum_{l=0}^{N-k} \sum_{d=1}^k  2 d \epsilon \int_{\tau}^{\tau+\delta} \bigl\lvert \kappa_\epsilon^L[c_r] (k\epsilon, l\epsilon, d\epsilon) \bigr\rvert \d r\\
&\le 4 \int_{\tau}^{\tau+\delta}  \sum_{k=1}^N \sum_{l=0}^{N-k}\sum_{d=1}^k   d\epsilon  (1+k\epsilon) (1+l\epsilon)\varphi(d\epsilon)\epsilon c_{k\epsilon} c_{l\epsilon}  \d r
\le 4  \bra*{1+\frac{N\epsilon}{L}}^2  C_{\varphi} \delta.
\end{split}\] 
We observe that $\delta \mapsto M_{\tau + \delta}^L(i\epsilon)- M_{\tau}^L(i\epsilon)$ is a $(F_{\tau+\delta})_{\delta\ge 0}$-martingale for any stopping time~$\tau$.
By the argument in \cite[Theorem 6.11 (Angle bracket process for solutions of martingale problems)]{eberle2023markov} and was also used in \cite[Proof of Theorem 4.11]{norris1999smoluchowski}, we have the equality
\[(M_{\tau+\delta}^L(i\epsilon) - M_\tau^L(i\epsilon))(M_{\tau+\delta}^L(j\epsilon) - M_\tau^L(j\epsilon))=N_{\delta}^{i,j} + \int_{\tau}^{\tau+ \delta} \Gamma_r^L(\operatorname{Proj}_{i\epsilon},\operatorname{Proj}_{j\epsilon})\d r,\] where $N^{i,j}$ is a martingale, $\operatorname{Proj}_{i\epsilon}(x) = x_{i\epsilon}$ as above and the $\Gamma$-operator is defined for linear functions $f,g: \bR^{\bN\epsilon} \to \bR$ by
\[
	\Gamma_r^L(f,g)=\Gamma^L(c_r) (f,g)=L^{-1}\sum_{k=1}^N \sum_{l=0}^{N-k} \sum_{d=1}^k  \kappa^L_\epsilon[c_r](k\epsilon,l\epsilon,d\epsilon)f(\gamma^{k\epsilon,l\epsilon, d\epsilon}) g(\gamma^{k\epsilon,l\epsilon, d\epsilon}) \,.
\] This operator is the Carré du champ operator of $\mathcal{Q}^{N,L,\epsilon}$. In particular, the quadratic variation  
$\pscal{M_{\tau+\cdot}^L(i\epsilon) - M_\tau^L(i\epsilon),M_{\tau+\cdot}^L(j\epsilon) - M_\tau^L(j\epsilon)}_\delta = \int_0^\delta \Gamma_{\tau+r}^L(\operatorname{Proj}_{i\epsilon},\operatorname{Proj}_{j\epsilon})\d r $.
Then  for each $m \in \{0,...,N\}$, we bound the quadratic term by
\[\begin{split} 
		\MoveEqLeft \biggl(\sum_{j=m}^N (M_{\tau+\delta}^L(j\epsilon) - M_{\tau}^L(j\epsilon))\biggr)^2
		= \sum_{i=m}^N  \sum_{j=m}^N  (M_{\tau+\delta}^L(i \epsilon)- M_\tau^L(i \epsilon))  (M_{\tau + \delta}^L(j \epsilon) - M_{\tau}^L( j \epsilon))\\
&=\sum_{i=m}^N \sum_{j=m}^N N_{\delta}^{i,j} +  L^{-1} \sum_{i=m}^N \sum_{j=m}^N \int_{\tau}^{\tau+ \delta}\sum_{k=1}^N \sum_{l=0}^{N-k} \sum_{d=1}^k  \kappa^L_\epsilon[c_r](k\epsilon,l\epsilon,d\epsilon) \gamma^{k\epsilon,l\epsilon, d\epsilon}_{i\epsilon} \gamma^{k\epsilon,l\epsilon, d\epsilon}_{j\epsilon}\d r.
\end{split}\] 
Above, we use~\eqref{eq:def:gamma} in the slightly abused form $\gamma^{x,y,z}_{l} = - \delta(x-l) - \delta(y-l) + \delta(x-z-l)+\delta(y+z-l)=\gamma^{x,y,z} \cdot \delta_{l} $.
The martingale property implies 
$\bE^{N,L,\epsilon}[ N^{i,j}_\delta ] =\bE^{N,L,\epsilon}[ N^{i,j}_0 ] = 0$. In the remaining part of the proof, we abbreviate the triple sum $\sum_{k=1}^N \sum_{l=0}^{N-k} \sum_{d=1}^k $ by $\sum_{(k,l,d)\in\Delta_N}$.
We obtain the pointwise estimate
\[ \begin{split}
	\MoveEqLeft
	L^{-1} \sum_{i=m}^N \sum_{j=m}^N \int_{\tau}^{\tau+ \delta}\sum_{\mathclap{(k,l,d)\in\Delta_N}}	 \kappa^L_\epsilon[c_r](k\epsilon,l\epsilon,d\epsilon) \gamma^{k\epsilon,l\epsilon, d\epsilon}_{i\epsilon} \gamma^{k\epsilon,l\epsilon, d\epsilon}_{j\epsilon}\d r\\
&\le L^{-1} \mkern-16mu\sum_{{(k,l,d)\in\Delta_N}} \Biggl|\sum_{i=m}^N \gamma^{k\epsilon,l\epsilon, d\epsilon}_{i\epsilon}\Biggr| \Biggl|\sum_{j=m}^N   \gamma^{k\epsilon,l\epsilon, d\epsilon}_{j\epsilon}\Biggr| \int_{\tau}^{\tau+ \delta}  \kappa^L_\epsilon[c_r](k\epsilon,l\epsilon,d\epsilon)  \dx r \\
&= L^{-1} \ \ \sum_{\mathclap{(k,l,d)\in\Delta_N}}\ \  \bigl|-\1_{(k-d,k]}(m) + \1_{(l,l+d]}(m)\bigr|^2 \int_{\tau}^{\tau+ \delta}  \kappa^L_\epsilon[c_r](k\epsilon,l\epsilon,d\epsilon)  \d r\\
& =L^{-1} \ \ \sum_{\mathclap{(k,l,d)\in\Delta_N}} \ \ \bigl\lvert -\1_{(k-d,k]}(m) + \1_{(l,l+d]}(m)\bigr\rvert \int_{\tau}^{\tau+ \delta}  \kappa^L_\epsilon[c_r](k\epsilon,l\epsilon,d\epsilon)  \d r.
\end{split}\]
By doing so, we can estimate the martingale term in \eqref{eq:martingale problem finite} using the above calculation and  Jensen's inequality to get
\[\begin{split}
		\MoveEqLeft \bE^{N,L,\epsilon}\biggl[ \sum_{m=0}^N \epsilon \Bigl|\sum_{j=m}^N (M_{\tau+\delta}^L(j\epsilon) - M_{\tau}^L(j\epsilon))\Bigr|\biggr]
	= \sum_{m=0}^N \epsilon \bE^{N,L,\epsilon}\biggl[  \Bigl(\sum_{j=m}^N (M_{\tau+\delta}^L(j\epsilon) - M_{\tau}^L(j\epsilon))\Bigr)^2 \biggr]^{1/2} \\
&\le  \epsilon \sum_{m=0}^N   \bE^{N,L,\epsilon}\biggl[  L^{-1} \sum_{i=m}^N \sum_{j=m}^N \int_{\tau}^{\tau+ \delta}\ \ \sum_{\mathclap{(k,l,d)\in\Delta_N}}\ \  \kappa^L_\epsilon[c_r](k\epsilon,l\epsilon,d\epsilon) \gamma^{k\epsilon,l\epsilon, d\epsilon}_{i\epsilon} \gamma^{k\epsilon,l\epsilon, d\epsilon}_{j\epsilon}\d r \biggr]^{\frac12} \\
&\le \epsilon \sqrt{N+1} \Biggl( \sum_{m=0}^N  \bE^{N,L,\epsilon}\biggl[  L^{-1} \sum_{\mathclap{(k,l,d)\in\Delta_N}}\ \bigl|-\1_{(k-d,k]}(m) + \1_{(l,l+d]}(m)\bigr| \int_{\tau}^{\tau+ \delta}  \kappa^L_\epsilon[c_r](k\epsilon,l\epsilon,d\epsilon)  \d r\biggr] \Biggr)^{\frac12}\\
&= \epsilon \sqrt{N+1} \bigg( \bE^{N,L,\epsilon}\bigg[  L^{-1} \sum_{\mathclap{(k,l,d)\in\Delta_N}}\ \   2d  \int_{\tau}^{\tau+ \delta}  \kappa^L_\epsilon[c_r](k\epsilon,l\epsilon,d\epsilon)  \d r\bigg] \bigg)^{\frac12}\\
& \le  \epsilon \sqrt{N+1} \bigg( \bE^{N,L,\epsilon}\bigg[  L^{-1} \epsilon^{-1} \int_{\tau}^{\tau+ \delta} \ \ \sum_{\mathclap{(k,l,d)\in\Delta_N}}\ \   2d \epsilon     (1+k\epsilon)(1+l\epsilon) \varphi(d \epsilon)\epsilon c_{r}(k\epsilon) c_r(l\epsilon)  \d r\bigg] \bigg)^{\frac12}\\
& \le  \epsilon \sqrt{N+1} \biggl(   L^{-1}  \epsilon^{-1} 2 \biggl(1+\frac{N\epsilon}{L}\biggr)^2 C_{\varphi} \delta \biggr)^{\frac12} = \sqrt{\frac{ (N+1)\epsilon}{L}} (2C_{\varphi}\delta )^{1/2} \biggl(1+\frac{N\epsilon}{L}\biggr) \,.
 \end{split}\]
where we used the bound~\eqref{eq:gamma:bound}. 
Since $\frac{N\epsilon}{L} \to \rho$ by assumption, we verify Adlous' criterion~\eqref{eq:AldousCriterion} by an application of the Markov inequality 
\[
\bP^{N,L,\epsilon}(c:W_1(c_{\tau+\delta}, c_{\tau})\ge \mu) \le \frac{1}{\mu}  \bE^{N,L,\epsilon}[W_1(c_{\tau+\delta}, c_{\tau}) ]\le \frac{C}{\mu} \delta^{\frac12}. \qedhere
\]
\end{proof} 
\begin{proposition}[Identification of limit]\label{prop:ident-limit}
	Let the assumptions of Proposition~\ref{prop:tightness in dexg} hold. 
	Let $\eps\to 0$ such that $\max\{\omega_{\overline{\varphi}}(\epsilon),\omega_K(\epsilon)\}L \to 0$ as $\frac{N\epsilon}{L}\to \rho$, where $\omega_K$ and $\omega_{\overline{\varphi}}$ are the moduli of continuity of $K$ and $(1+x+f(x))\varphi(x)$ or $(1+x) \varphi(x)$ depending on the cases in  Proposition~\ref{prop:tightness in dexg}.
	Then up to a subsequence $\bP^{N,L,\epsilon} \to \delta_c$ in $D([0,\infty),\mathcal{P}^1)$where $c$ is a weak solution to~\eqref{eq:CGEDG} in the sense of Definition~\ref{def:weak-sol}.
\end{proposition}
\begin{proof}Under the assumptions of Proposition \ref{prop:tightness in dexg}, we have the tightness of $(\bP^{N,L,\epsilon})$, measures on the Skorokhod space $D([0,\infty), V)$ with $(V,d)= (\cP^1,W_1)$. By Prokhrov theorem, there exists a stochastic process $c \in D([0,\infty), \cP_{\le \overline{\rho}}(\bR_+) )$ with $\overline{\rho}= \sup \frac{N\epsilon}{L}$ with its distribution being a limit of weak convergence of probability measures denoted $\bP$. 
We want to show that each of the three terms in the equation 
\begin{equation}\label{eq:prelimit evolution}c_{t}^L- c_{s}^L = \int^{t}_{s} \mathcal{Q}^{N,L,\epsilon}c_r^L \dx r+ (M_{t}^L- M_s^L),\quad t \ge s \ge 0,\end{equation}
converge in law if $\bP^{N,L,\epsilon} \to \bP$ where $c^L \sim \bP^{N,L,\epsilon}$ and the limit process $c \sim \bP$ satisfies  the deterministic equation 
\[c_{\tau +\delta } - c_{\tau} = \int_{\tau}^{\tau+\delta} Q(c_r) \d r \] 
for some operator $Q$ in duality with Lipschitz functions that will be identified as the generator from~\eqref{eq:def:generator}.
For $i \in \{0,1,...,N\}$, we compute
\[\begin{split}\bE^{N,L,\epsilon} [(M_{t}^L(i\epsilon) - M_{s}^L(i\epsilon))^2]&=  \bE^{N,L,\epsilon}\bigg[   L^{-1}  \int^{t}_{s}\sum_{(k,l,d)\in\Delta_N} \kappa^L_\epsilon[c_r](k\epsilon,l\epsilon,d\epsilon) (\gamma^{k\epsilon,l\epsilon, d\epsilon}_{i\epsilon})^2 \d r\bigg]\\
 &\le  \bE^{N,L,\epsilon}\bigg[2 (2)^2  L^{-1} \biggl(1+\frac{N\epsilon}{L}\biggr)^2 C_{\varphi} (t-s) \bigg] \le \frac{C(t-s)}{L}. 
\end{split}\]
The estimate above is uniform in initial condition and depends on time only through $(t-s)$. By Doob's martingale inequality, 
\[
	\bE^{N,L,\epsilon}\Bigl[\sup_{0\le s \le T} M_s^{L}(i\epsilon)^2\Bigr]
	\le 4  \bE^{N,L,\epsilon}\bigl[M_T^{L}(i\epsilon)^2\bigr] \qquad\text{for $T>0$.}
\]
This implies $M_{\cdot}^L(i\epsilon)  \to  0$ almost surely as $\bP^{N,L,\epsilon} \to \bP$ uniformly on compact interval of time, for each $i \in \{0,1,...,N\}$. 

Let $T_\bP$ be the set of times $t$ for which time projection $\pi_t(c)=c_t$ is continuous outside on a set of $\bP$-measure $0$. Then for $t,t+\delta \in T_\bP$, we have $\bP^{N,L,\epsilon}\pi^{-1}_{t,t+\delta} \to \bP\pi^{-1}_{t,t+\delta}$ by \cite[Theorem 16.7]{billingsley2013convergence}. 
Since, we have for Lipschitz functions $f$ on $\bR_+$, that 
\[\int f \d (c -c') \le \Lip(f) W_1(c,c') \,, 
\]
we get that map $c\mapsto \int f \d c$ is continuous with respect to $W_1$.
Hence, we can apply the continuous mapping theorem, to get for Lipschitz continuous functions $f$  on $\bR_+$ the convergence of $\int f (c^{L}_{t+\delta} -c^L_{t})$  to $\int f (c_{t+\delta}-c_t)$ in law where $c_t \sim \bP\pi^{-1}_{t}$.

It remains to show the convergence in law of the integral on the right side of \eqref{eq:prelimit evolution} when tested with respect to Lipschitz function $f$. We abbreviate $I^L_{s,t}(c)=\int_{s}^{t} F^L(c_r) \d r$,  for $c \in \hat{V}^{N,L,\epsilon}$,  
where  
\[
\begin{split}F^L(c) =\int_{\bR_+} f(x)  \mathcal{Q}^{N,L,\epsilon} c(\d x) = \sum_{k=1}^N  \sum_{l=0}^{N-k} \sum_{d=1}^k \kappa_\epsilon^L[c] ( k \epsilon, l \epsilon, d \epsilon) \  \gamma^{k\epsilon,l\epsilon, d \epsilon}\cdot f  ,
\end{split} 
\] 
with $\gamma^{x,y,z} \cdot f =(-f(x) - f(y) + f(x-z) + f(y+z)$. 
The tentative limit is $I_{s,t}: c\mapsto \int_s^t F(c_r ) \d r $, where for $c \in \cP^1(\bR_+)$, we define in terms of the generator from~\eqref{eq:def:generator} the function
\[
\begin{split}
	F(c)&=  Q(c) \cdot f = \int_0^\infty \int_0^\infty   \int_0^x  K(x,y,z)  \gamma^{x,y,z}\cdot f \,   \d z\,c(\d y)c(\d x).
\end{split}\]
By using that $f$ is Lipschitz and Assumption~\ref{ass:exist}, we can bound
\[
\begin{split}
F(c) &\le  2 \Lip(f)  \int_0^\infty  \int_0^\infty  K(x,y) \int_0^x\varphi(z)  z  \, \d z\, c(\d y)c(\d x) \\
&\le 2 \Lip(f)  \|\varphi\|_{L^{1,1}} \int_0^\infty  \int _0^\infty K(x,y)  c (\d y) c (\d x)\\
&\le 2 C \Lip(f)  \|\varphi\|_{L^{1,1}} (1+\cM_1(c))^2.
\end{split}\]
The convergence of the random variable $I_{s,t}^L[c^L] \to I_{s,t}[c]$ in law is equivalent to the convergence 
for all Lipschitz function $h$ of 
\[
	\tilde\bE\bigl[ h(I_{s,t}^L[c^L])  - h( I_{s,t}[c])\bigr] \to 0 \quad\text{ as }\quad L \to \infty,
\]
where $\tilde\bE$ is with respect to $\tilde\bP$ such that $\tilde\bP_\#c^{L}= \bP^{N,L,\epsilon}$ and $\tilde\bP_\#c= \bP$.
We consider the decomposition
\begin{equation}\label{eq:deomposition}
	\tilde\bE\bigl[ h(I_{s,t}^L[c^L])  - h( I_{s,t}[c])\bigr] = \tilde\bE\bigl[ h(I_{s,t}^L[c^L])  -  h(I_{s,t}[c^L])\bigr] + \tilde\bE\bigl[ h( I_{s,t}[c^L]) - h(I_{s,t}[c]) \bigr] \,.
\end{equation}
We begin with the first term in the decomposition~\eqref{eq:deomposition}, where by Lipschitz continuity, it suffices to show $\tilde\bE[|I_{s,t}^L[c^L]- I_{s,t}[c^L] |]\to 0$ as $L \to \infty$.
We have for fixed $c^L\in \hat{V}^{N,L,\epsilon}$,  the estimate
\[\begin{split}
	\MoveEqLeft |F(c^L) - F^L(c^L)| = \bigg|\sum_{k=1}^N  c^L(k \epsilon) \sum_{l=k}^N c^L((l-k) \epsilon) \bigg(\int_0^{k \epsilon} K(k\epsilon, (l-k)\epsilon, z)  f\cdot \gamma^{k\epsilon, (l-k)\epsilon,z} \d z \\
&\omit\hfill$\displaystyle -  \epsilon \sum_{d=1}^k K(k \epsilon, (l-k)\epsilon, d\epsilon)  f \cdot \gamma^{k\epsilon, (l-k)\epsilon, d \epsilon}   \bigg) \bigg| + O(1/L)$
\\
&=\Bigg| \sum_{k=1}^N c^L(k \epsilon) \sum_{l=0}^{N-k} c^L(l\epsilon)  \sum_{d=1}^k  \int_{(d-1)\epsilon}^{d \epsilon}  \bigg(  (K(k\epsilon, l\epsilon, z) - K(k\epsilon, l\epsilon, d \epsilon)) f \cdot \gamma^{k\epsilon, l\epsilon, z} \\
&\omit\hfill$\displaystyle +  K(k\epsilon, l\epsilon, d\epsilon) (f \cdot \gamma^{(l+d)\epsilon, (k-d)\epsilon, d\epsilon -z}) \bigg)\d z \Bigg|+ O(1/L)$ \\
%& \le  2 \Lip(f) \sum_{k=1}^N c^L(k\epsilon) \sum_{l=k}^N c^L((l-k)\epsilon) \sum_{d=1}^k \\
%&\bigg(\int_{(d-1)\epsilon}^{d \epsilon} \lambda (\d z) |K(k\epsilon, (l-k)\epsilon, z) - K(k\epsilon, (l-k)\epsilon, d \epsilon)|  z \\
%& + \epsilon K(k\epsilon, (l-k)\epsilon, d \epsilon) | d\epsilon -z|\bigg) + O(1/L)\\
& \le  2 \Lip(f) \sum_{k=1}^N c^L(k\epsilon) \sum_{l=0}^{N-k} c^L(l\epsilon) \sum_{d=1}^k \int_{(d-1)\epsilon}^{d \epsilon} \bigg(  \omega_K(\epsilon) z +   K(k\epsilon, l\epsilon, d \epsilon) (d\epsilon -z)    \bigg) \d z+ O(1/L)\\
& \le  \Lip(f) \sum_{k=1}^N c^L(k\epsilon) \sum_{l=0}^{N-k} c^L(l\epsilon)  \bigg( \omega_K (\epsilon) (k\epsilon)^2+ (1+k\epsilon)(1+l\epsilon)\sum_{d=1}^k  \varphi(d \epsilon)\epsilon^2 \bigg) +O(1/L)   \\
& \le \Lip(f) \sum_{k=1}^N c^L(k\epsilon) \sum_{l=0}^{N-k} c^L(l\epsilon)  \bigg( \omega_K (\epsilon)(N\eps) (k\epsilon)+ (1+k\epsilon)(1+l\epsilon)\sum_{d=1}^k  \varphi(d \epsilon)\epsilon^2 \bigg) +O(1/L)\\
 & \le \Lip(f) \bigg( \overline\rho \omega_K(\eps)L \cM_1(c^L)  +\epsilon (1+\cM_1(c^L) )^2  \|\varphi\|_{L^1} \bigg) +O(1/L)
\end{split}\]
where we used $\omega_K(\epsilon) (N\epsilon) \le \overline{\rho} \omega_K(\epsilon)  L  \to 0$ 
 %so that $ \omega_K(\epsilon)  N \epsilon = o(1)_{N\epsilon/L\to \rho}$ 
and Lemma \ref{lem:compactness via second moment}.

Let $I_{s,t}: c\mapsto \int_s^t F(c_r ) \d r $. The calculation above shows that $I_{s,t}$ is bounded on probability measures with uniformly bounded first moment. We note
\begin{equation*}
 \quad |I_{s,t}(c^L)-I^L_{s,t}(c^L)| \le  |t-s|\sup_{r\in[s,t]}|F(c^L_r) - F^L(c^L_r)|.
\end{equation*}  
 
Since the first moment is uniformly bounded by $\overline\rho$ along the sequence $N\epsilon/L\to \rho$, we conclude that 
\[
	\lim_{N\epsilon/ L \to \rho}   |I_{s,t}(c^L)-I^L_{s,t}(c^L)|  =0 \qquad\text{ uniformly for } c^L \in \hat{V}^{N,L,\epsilon} \,.
\] 
It remains to show the convergence of the second term in the decomposition~\eqref{eq:deomposition}, that is for all $h$ Lipschitz continuous 
\[\tilde\bE\bigl[ h(I_{s,t}[c])  -  h(I_{s,t}[c^L]) \bigr]  \to 0 \qquad \text{ as } L \to \infty.\]
We define for any $R>0$ fixed that truncated limit functional $\overline I_{s,t}^R(c)= \int_s^t F_R(c_r) \d r$ with
 
\[ F_R (c) =\int_{[0,R]^2}  \int_0^x   K(x,y,z) \gamma^{x,y,z}\cdot f \d z\, c( \d y)  c(\d x). \]  With this definition, our strategy is to decompose the difference further and control each of the three terms on the right-hand side in
\[\begin{split}\tilde\bE\bigl[ h(I_{s,t}[c])  -  h(I_{s,t}[c^L]) \bigr]&=\tilde\bE\bigl[ h(I_{s,t}[c])  -  h(\overline I_{s,t}^R[c]) \bigr] \\
&\quad + \tilde\bE\bigl[ h(\overline I_{s,t}^R[c])- h(\overline I_{s,t}^R[c^L] )\bigr] + \tilde\bE\bigl[ h(\overline I_{s,t}^R[c^L])- h(I_{s,t}[c^L] )   \bigr].\end{split}\] 
By continuity of $K$, $K(\cdot,\cdot,z)$ is bounded on $[0,R]^2$ for $z\ge 0$. The boundedness and continuity of $K$ imply that the map 
\[[0,R]^2 \ni(x,y) \mapsto  \int_0^x K(x,y,z) \gamma^{x,y,z}\cdot f \d z \] 
is continuous by dominated convergence.  
If $c^n \to c$ in $W_1$, we get also $c^n \times c^n \to c \times c$ converges weakly and so $F_R(c^n)\to F_R(c)$.
This shows that $F_R$ is continuous with respect to $W_1$ for each $R>0$.
By dominated convergence in time, we have $\overline I^R_{s,t}$ is continuous on $D([0,\infty), (\cP(\bR_+),W_1))$. Therefore, if $c^n \to c $ in law on $D([0,\infty), (\cP(\bR_+),W_1))$, the continuous mapping theorem implies the sequence of random variables $\overline I^R_{s,t}(c^n) \to \overline I^R_{s,t}(c)$ in law.  This implies for all $R>0$, all Lipschitz continuous and bounded functions $f$,  \[\tilde\bE\bigl[ h(\overline I_{s,t}^R[c])- h(\overline I_{s,t}^R[c^L] )\bigr] \to 0 \qquad \text{ as } L \to \infty.\]
For the other two terms, the idea is to use the bounded superlinear moment $\cM_g$ of $c$ and~$c^L$ in expectation with $g\in \cA$ from Proposition \ref{prop: propa of 2. moment}. For $g \in \cA$ and $c \sim \hat{\bP}\in \{ \bP^{N,L,\epsilon},\bP\}$ on $D([0,\infty), (\cP_{\le \overline{\rho}}(\bR_+),W_1))$, we take the expectation and estimate
\[\begin{split}
\frac{1}{\operatorname{Lip}(h)} \hat{\bE}\bigl[ h(I_{s,t}[c])  -  h(\overline I_{s,t}^R[c])\bigr] &\le  \hat{\bE}\Bigl[\bigl|I_{s,t}[c] - \overline I_{s,t}^R[c]\bigr|\Bigr] \\
&\le \hat{\bE}\biggl[\int_s^t |F(c_r)-F_R(c_r)| \d r \biggr] \\
&= \hat{\bE}\biggl[\int_s^t \biggl| \iint_{{[0,R)^2}^c} \int_0^x   K(x,y,z) \gamma^{x,y,z}\cdot f \d z \, c_r (\d y)  c_r (\d x) \biggr| \d r \biggr]\\
&\le 2 \operatorname{Lip}(f)  \overline{\rho} C_\varphi  \hat{\bE}\biggl[\int_s^t   \int_{R}^{\infty}(1+x)   c_r (\d x)  \d r \biggr]\\ 
& \le 2 \operatorname{Lip}(f)  \overline{\rho} C_\varphi \frac{1+R}{g(R)} \hat{\bE}\biggl[\int_s^t   \int_{R}^{\infty}g(x)   c_r (\d x)  \d r \biggr]\\ 
&  \le 2 \operatorname{Lip}(f)  \overline{\rho} C_\varphi \frac{1+R}{g(R)} \hat{\bE}\biggl[\int_s^t  \cM_g(c_r) \d r \biggr]\\ 
&= 2 \operatorname{Lip}(f)  \overline{\rho} C_\varphi \frac{1+R}{g(R)} \int_s^t \hat{\bE}_r [\cM_g(c_r) ] \d r \\ 
&\le 2  \operatorname{Lip}(f)    \overline{\rho} C_\varphi |t-s|  \frac{1+R}{g(R)} \sup_{r\in[s,t]}\hat{\bE}_r [\cM_g(c_r) ] \,,
\end{split}
\] where $\hat{\bE}_r $ is with respect to $\hat{\bP}_{\#}\pi_t$ with $\pi_t$ being the  time projection at time $t$.\\
%Above we used the disintegration of measure.
Together with Proposition \ref{prop: propa of 2. moment}  and superlinearity of  $g\in \cA$, the above estimate implies that along a converging sequence $\bP^{N,L,\epsilon} \to \bP$ on $D([0,\infty), \cP_{\le \overline{\rho}}(\bR_+))$, that 
\[
	\bE^{N,L,\epsilon}\bigl[ |I_{s,t}[c] - \overline I^R_{s,t}[c]|\bigr] =\tilde\bE\bigl[ |I_{s,t}[c^L] - \overline I^R_{s,t}[c^L]|\bigr] \to 0
	\quad 
	\text{ as $R \to \infty$ uniformly in $N,L, \epsilon$.}
\] 
Hence, by lower semicontinuity also
\[\sup_{t\in [0,T] }\bE \bigl[\cM_{g}(c_t)\bigr]\le  \sup_{t\in [0,T] }  \liminf_{N\epsilon/L \to \rho}\bE^{N,L,\epsilon}\bigl[\cM_g(c_t)\bigr] \,,
\] 
and we conclude 
\[
	\tilde\bE\bigl[ |I_{s,t}[c] - \overline I^R_{s,t}[c]|\bigr] \to 0 \qquad\text{ as } R\to \infty \,.
\]
Therefore for all bounded Lipschitz continuous functions $h$, $\delta >0$, we can choose $R$ large enough such that 
$\max\Big(\tilde\bE\bigl[ \big|I_{s,t}[c^L] - \overline I^R_{s,t}[c^L]\bigr|\bigr] ,\tilde\bE\bigl[ \bigl|I_{s,t}[c] - \overline I^R_{s,t}[c]\bigr|\bigr] \Big)< \delta /3 $, and $L(R)$ large enough such that $\bigl| \tilde\bE\bigl[ h(\overline I_{s,t}^R[c])- h(\overline I_{s,t}^R[c^L] )\bigr]\bigr| < \delta /3 $.
With this choice for $L$, we conclude
\[\bigl|\tilde\bE\bigl[ h(I_{s,t}[c])  -  h(I_{s,t}[c^L]) \bigr]\bigr| < \delta .\] Since $\delta>0$ is arbitrary, we showed 
$\lim_{L\to \infty}\tilde\bE\bigl[ h(I_{s,t}[c])  -  h(I_{s,t}[c^L])\bigr]=0 $ for each $h$ that is bounded Lipschitz continuous. This concludes the proof of the convergence in law  $I_{s,t}^L[c^L] \to I_{s,t}[c]$.
\end{proof}

\section{Lebesgue density for the limit}

In this section we show that under suitable assumptions on the initial data, the limit measure is absolutely continuous with respect to the Lebesgue measure.  The argument is similar to \cite[Theorem 1.2]{yaghouti2009coagulation} and \cite[Lemma 6.3]{GuoPapnicolaouVaradhan1988DiffusionLimit} via the control of the growth of relative entropy with respect to a product measure. We use a discrete approximation of exponential distribution as the product measure, because it allows cancellations in the exchange gradient.
%\begin{proposition}Suppose the law  $\bP^{N,L,\epsilon}=C^L_{\#}\mu_{N,L,\epsilon}$ is induced by a microscopic evolution where $\mu_{N,L,\epsilon}(t,\eta) \in \cP(V^{N,L,\epsilon})$ is the law of a Markov process with the microscopic generator $\cL^{N,L,\epsilon}$ and $\bP^{N,L,\epsilon}\to \bP$ weakly in  $D([0,\infty), \cP_{\le\overline{\rho}})$. Moreover assume for some $b>0$
%\begin{equation}\label{eq:entropy-bdd}\sup_{N\epsilon/ L \to \rho}\frac1L H\bra*{\mu_{N,L,\epsilon}(0) | \nu_{L,\epsilon,b} }<+\infty,\end{equation}
%where $\nu_{L,\epsilon,b} \in \cP(V^{N,L,\epsilon})$,  $V^{N,L,\epsilon}=\{ \eta \in (\bN_0 \epsilon)^{L}   \}$, is the reference measure defined by $\nu_{L,\epsilon,b}(A) =\sum_{\eta \in A} \Phi_{L,\epsilon,b}(\eta)= \sum_{\eta\in A}\prod_{i=1}^L g_{\epsilon,b}(\eta_i)$, \[g_{\epsilon,b}(m\epsilon) =b \int_{m\epsilon}^{(m+1)\epsilon} e^{-b x} \d x,\quad b^{-1} =\sum_{i\in\bN_0} \int_{i \epsilon}^{(i+1)\epsilon} e^{-b x} \d x.\]
%Then for $\bP$ is concentrated on paths $(c_t)_{t\ge0}$ such that the probability measure $c_t$ on $\bR_+$ has a Lebesgue density for Lebesgue almost all $t$. 
%\end{proposition}

We start with a Lemma showing that the entropy bound~\eqref{eq:entropy-bdd} for the initial data is propagated in time with an at most linear growth.
\begin{lemma}[Propagation of entropy bound]\label{lem:entropy:bound}
	Let $\mu_{N,L,\eps}$ and $\nu_{L,\eps,b}$ be as in the setting of Proposition~\ref{prop:abscty of soln}.
	Let $\rho\geq 0$ and any sequence $\frac{N\eps}{L}\to\rho$ with $\overline{\rho}=\sup_{N\epsilon/L \to \rho} \frac{N\epsilon}{L}$ then there exists $C=C(K,\overline\rho)>0$ such that for any $0\leq s \leq t$, it holds
	\begin{equation}\label{eq:entropy-growth}
		\limsup_{\frac{N\eps}{L} \to \rho} \frac{1}{L} \bra*{H\bigl(\mu_{N,L,\epsilon}(t) | \nu_{L,\epsilon,b} \bigr) - H\bigl(\mu_{N,L,\epsilon}(s) | \nu_{L,\epsilon,b} \bigr) } \leq C(t-s) \,.
	\end{equation}
\end{lemma}
\begin{proof}
	%By construction, we note that 
%	\[\sum_{\eta\in V^{N,L,\epsilon}}\nu_{L,\epsilon,b}(\eta) = \prod_{i=1}^L \sum_{\eta_i \in \bN_0 \epsilon} b \int^{\eta_i +\epsilon}_{\eta_i} e^{-b x} \d x = 1 \,.\]
	%We replicate the arguments in \cite{GuoPapnicolaouVaradhan1988DiffusionLimit} so we show the relative entropy scales with $N$. 
	%Let $(\mu_{N,L,\epsilon}(t))_{t\ge 0} \subset \cP(V^{N,L,\epsilon})$ be  the law of the Markov process with infinitesimal generator $\cL^{N,L,\epsilon}$ 
	We denote with $f_{N,L,\epsilon}(t)=\frac{\d \mu_{N,L,\epsilon}(t)}{\d \nu_{L,\epsilon,b}}$ the density with respect to $\nu_{L,\epsilon,b}$. The relative entropy is given by
	\[ H(\mu_{N,L,\epsilon}(t) | \nu_{L,\epsilon,b} ) = \sum_{\eta\in V^{N,L,\epsilon}} \nu_{L,\epsilon,b}(\eta) f_{N,L,\epsilon}(t,\eta) \log \bigl(f_{N,L,\epsilon}(t,\eta)\bigr).\] 
	We can rewrite time differences of the relative entropy in terms of the generator from~\eqref{eq:def:generator:micro} and get
	\[H\bigl(\mu_{N,L,\epsilon}(t) | \nu_{L,\epsilon,b} \bigr) - H\bigl(\mu_{N,L,\epsilon}(s) | \nu_{L,\epsilon} \bigr)=\int_s^t  \sum_{\eta\in V^{N,L,\epsilon}}\nu_{L,\epsilon,b}(\eta) f_{N,L,\epsilon}(r,\eta)\cL^{N,L,\epsilon} \log f_{N,L,\epsilon}(r,\eta) \d r \,.
	\] 
	By using $\log x \le x$, we can for a.e. $r\in (s,t)$ and all $\eta\in V^{N,L,\eps}$ such that $f_{N,L,\epsilon}(r,\eta)>0$ estimate
	\[\begin{split}\cL^{N,L,\epsilon} \log f_{N,L,\epsilon}(r,\eta) &= \frac{1}{L-1} \sum_{x,y=1}^L \sum_{d=1}^{\eta_x / \epsilon} K_{\epsilon}(\eta_x ,\eta_y,d\epsilon) \bigl(\log f_{N,L,\epsilon}(r,\eta^{x\stackrel{d\epsilon}{\to} y}) - \log f_{N,L,\epsilon}(r, \eta) \bigr) \\
		&\le  \frac{1}{L-1} \sum_{x,y=1}^L \sum_{d=1}^{\eta_x / \epsilon} K_{\epsilon}(\eta_x ,\eta_y,d\epsilon) \frac{f_{N,L,\epsilon}(r,\eta^{x\stackrel{d\epsilon}{\to} y})}{ f_{N,L,\epsilon}(r, \eta)} \,,
	\end{split}
	\] 
	where $\eta^{x\stackrel{d\epsilon}{\to} y}_z$ denotes the state after one transition defined in~\eqref{eq:def:jump}.
	% =\begin{cases} \eta_x - d\epsilon \text{ if } z=x,\\
		% \eta_y + d\epsilon \text{ if } z= y,\\
		% \eta_z \text{ otherwise.}\end{cases}
	%$
	Next, we observe that by the definition of $g_{\epsilon,b}$ in~\eqref{eq:def:g:measure}, we have the identity
	\[\begin{split}\gamma^{y,x,z}\cdot \log g_{\epsilon,b} &= -\log g_{\epsilon,b}(x)-\log g_{\epsilon,b}(y)+\log g_{\epsilon,b}(x+z)+\log g_{\epsilon,b}(y-z) \\
		&= \log \bigl( e^{-b\tilde x}\bigr\vert^{y-z + \epsilon}_{y-z}  e^{-b\tilde x}\bigr \vert^{x+z + \epsilon}_{x+z}\bigr) -\log\bigl( e^{-b\tilde x}\bigr\vert^{y + \epsilon}_{y} e^{-b\tilde x}\bigr\vert^{x + \epsilon}_{x} \bigr) \\
		&= \log \frac{e^{-b(y-z)}e^{-b(x+z)}(e^{-b \epsilon }-1)^2}{e^{-by}e^{-bx}(e^{-b \epsilon} -1)^2 } 
		=\log 1 =0.
	\end{split}\]
	In particular, we get for any admissible transition the identity
	\[
		\nu_{L,\epsilon,b}(\eta^{y\stackrel{d\epsilon}{\to} x}) = \nu_{L,\epsilon,b}(\eta) \exp(-b \, \gamma^{\eta_y,\eta_x,d\epsilon} \cdot \log g_{\epsilon,b}) = \nu_{L,\epsilon,b}(\eta)    \,.
	\]
	This implies
	\[\begin{split}
		\MoveEqLeft H\bigl(\mu_{N,L,\epsilon}(t) | \nu_{L,\epsilon,b} \bigr) - H\bigl(\mu_{N,L,\epsilon}(s) | \nu_{L,\epsilon,b} \bigr) \\
		&\le \int_s^t \sum_{\eta\in V^{N,L,\epsilon}} \nu_{L,\epsilon,b}(\eta)   \frac{1}{L-1} \sum_{x,y=1}^L \sum_{d=1}^{\eta_x / \epsilon} K_{\epsilon}(\eta_x ,\eta_y,d\epsilon) f_{N,L,\epsilon}(r,\eta^{x\stackrel{d\epsilon}{\to} y}) \d r \\
		%	&=\int_s^t \frac{1}{L-1}\sum_{\eta\in V^{N,L,\epsilon}}\sum_{x,y=1}^L \sum_{d=1}^{\eta_y / \epsilon} \nu_{L,\epsilon,b}(\eta^{y\stackrel{d\epsilon}{\to} x}) K_\epsilon(\eta_x + d \epsilon , \eta_y - d \epsilon, d \epsilon) f_{N,L,\epsilon}(r, \eta)  \d r \\
		%<	&=\int_s^t  \frac{1}{L-1}\sum_{\eta\in V^{N,L,\epsilon}}\sum_{x,y=1}^L \sum_{d=1}^{\eta_x / \epsilon} \nu_{L,\epsilon,b}(\eta) \exp(\gamma^{\eta_y,\eta_x,d\epsilon} \cdot \log g_{\epsilon,b}) K_\epsilon(\eta_x + d \epsilon , \eta_y - d \epsilon, d \epsilon) f_{N,L,\epsilon}(r, \eta)  \d r\\
			&=\int_s^t \frac{1}{L-1}\sum_{\eta\in V^{N,L,\epsilon}}\sum_{x,y=1}^L \sum_{d=1}^{\eta_y / \epsilon} \nu_{L,\epsilon,b}(\eta)  K_\epsilon(\eta_x + d \epsilon , \eta_y - d \epsilon, d \epsilon) f_{N,L,\epsilon}(r, \eta)  \d r\\
			&\stackrel{\mathclap{\eqref{z-bdd}}}{\le} C_K \int_s^t \frac{1}{L-1}\sum_{\eta\in V^{N,L,\epsilon}} \!\! \mu_{N,L,\epsilon}(r,\eta) \sum_{x,y=1}^L \sum_{d=1}^{\eta_x / \epsilon}    (1+\eta_x +d \epsilon)(1+\eta_y -d \epsilon) \varphi(d\epsilon)\epsilon  \d r\\
			&\le C_K \int_s^t  \frac{1}{L-1} \sum_{\eta\in V^{N,L,\epsilon}}  \mu_{N,L,\epsilon}(r,\eta) \sum_{x,y=1}^L \sum_{d=1}^{\eta_x / \epsilon}  \bigl((1+\eta_x)(1+\eta_y) + d\epsilon (1+\eta_y\bigr)) \varphi(d \epsilon) \epsilon \d r \\
			&\le C_K \int_s^t \frac{1}{L-1} \sum_{\eta\in V^{N,L,\epsilon}} \mu_{N,L,\epsilon}(r,\eta) \biggl[  (L+N\epsilon )^2 \sum_{d=1}^N \varphi(d \epsilon) \epsilon + L (L+N\epsilon) \sum_{d=1}^N d\epsilon  \varphi(d \epsilon) \epsilon\biggr] \d r \\
			&\le C_K  \frac{(L+N\epsilon )^2}{L-1}  \int_s^t \sum_{\eta\in V^{N,L,\epsilon}}\!\! \mu_{N,L,\epsilon}(r,\eta) \dx r  \, \eps \sum_{d=1}^N (1+d\eps) \varphi(d\eps)  \,.
			%\leq C_K C_{\varphi} C_{\overline{\rho}} (t-s) L \,. \qedhere
			%& \le C_K C_{\varphi} C_{\overline{\rho}}L \int_s^t \mu_{N,L,\epsilon}(r)(V^{N,L,\epsilon})   \d r
			%= C_K C_{\varphi} C_{\overline{\rho}} (t-s) L .
	\end{split}\]
	The conclusion follows by noting that $\frac{L+N\epsilon }{L}\leq 1+\bar\rho$, that $\mu_{N,L,\epsilon}(r)(V^{N,L,\eps}) = 1$ for almost every $r\in (s,t)$ and that $\limsup_{\eps\to0} \eps \sum_{d=1}^N (1+d\eps) \varphi(d\eps)<\infty$ thanks to Assumption~\ref{ass:exist}.
\end{proof}
The entropy bound~\eqref{eq:entropy-growth} from Lemma~\ref{lem:entropy:bound} is with a Sanov's large deviation principle the essential ingredient for the proof of Proposition~\ref{prop:abscty of soln}.
\begin{proof}[Proof of Proposition~\ref{prop:abscty of soln}]
The proof is based on showing a relative entropy bound for the time marginals with respect to the exponential measure $\lambda_b(\dx{x}) = b e^{-bx} \dx{x}$ on $\bbR_+$, that is for some $C>0$
\begin{equation}\label{eq:entropy:bound}
	\bE_t\bigl[ H(\cdot|\lambda_b )\bigr] \le C(1+t) \qquad\text{ for almost all $t\geq 0$,}
\end{equation} where $\bE_t$ is with respect to $\bP_\#\pi_t$ with $\pi_t$ being the time projection at time $t$.

The estimate~\eqref{eq:entropy-growth} implies for Lebesgue almost all $t\geq 0$, that $\bP_t$ is concentrated on probability measures on~$\bbR_+$ being absolutely continuous with respect to the exponential measure and in particular the Lebesgue measure on $\bR_+$.

For doing so, we consider the random empirical measure  $C^L(\eta) = \frac{1}{L} \sum_{x=1}^L \delta_{\eta_x} \in \cP(\bR_+)$ for $\eta \sim \nu_{L,\epsilon,b}$. 
By viewing $\mathcal{P}(\eps \bbN_0)\subset \calP(\bbR_+)$, it can be shown readily that $g_{\eps,b}$ defined in~\eqref{eq:def:g:measure} satsifies $g_{\epsilon,b}(\cdot ) \to \lambda_b$ weakly as probability measures on $\cP(\bR_+)$ for $\epsilon \to 0$, that is with respect to bounded uniformly continuous functions on $\bR_+$.

Then, the law $W_{L,\epsilon}= C^L_{\#}\nu_{L,\epsilon,b} \in \cP(\cP(\bN_0 \epsilon)) \subset \cP(\cP(\bR_+))$ satisfies a large deviation principle with speed $L$ and rate $H(\cdot| \lambda_b)$ by a generalized Sanov's theorem in \cite[Theorem 6.1]{Mariani2018}. By Varadhan's lemma \cite[Theorem III.13]{hollander2000large},  for any $u$ bounded continuous functions on $\cP(\bR_+)$, we have along a sequence of $\epsilon(L)\to 0$ as $L \to \infty$
\begin{equation}\label{eq:Sanov}
	\lim_{L\to \infty} \frac{1}{L} \log \sum_{\mathclap{\eta\in V^{N,L,\epsilon}}} e^{L u(C^L(\eta))} \nu_{L,\epsilon,b}(\eta) = \lim_{L\to \infty} \frac{1}{L} \log \sum_{\mathclap{c\in \hat{V}^{N,L,\epsilon}}} e^{L u(c)} W_{L,\epsilon}(c) = \sup_{\omega \in \cP(\bR_+)} \bigl\{u(\omega) - H(\omega | \lambda_b) \bigr\} \,. 
\end{equation}
We construct a sequence $(u_{k,j})$ of bounded continuous functions by fixing for each $k\geq 0$ a countable dense subset $(f_{k,i})_{i\in \bN}$ of $C([0,k])$ by the Weierstrass approximation theorem.  With that, we define for $\omega \in \cP(\bR_+)$
\[
	u_{k,j}(\omega)= \sup_{1\le i\le j} \biggl\{ \int f_{k,i}(x) \omega(\d x) - \log \int e^{f_{k,i}(x)} \lambda_b(\d x)\biggr\} \,.
\] 
We may assume that $f_{k,1}$ is the constant function $1$ on $[0,k]$ so that $u_{k,j}\ge -1$ for $k\ge 0, j\in\bN$. In addition, we also define $u_k : \cP(\bR_+) \to \bR$ by
\begin{equation*}%\label{def: uk}
	u_k(\omega)= \sup_{ i\in \bN}\biggl\{ \int f_{k,i}(x) \omega(\d x) - \log \int e^{f_{k,i}(x)} \lambda_b(\d x)\biggr\} \,.
\end{equation*} 
This choice is motivated by the variational characterization of the relative entropy \cite[Theorem 3.3]{rezakhanlou2015lectures}, which allows for any $\omega\in \cP(\bR_+)$ the estimate
\[\begin{split}
	H(\omega| \lambda_b) &= \sup_{f \in C_b ([0,\infty))}\biggl\{ \int f(x) \omega(\d x) - \log \int e^{f(x)} \lambda_b(\d x)\biggr\} \\
%	&\ge \sup_{f \in C_c ([0,\infty))}\biggl\{ \int f(x) \omega(\d x) - \log \int e^{f(x)} \lambda_b(\d x\biggr\}\\
%	&=\sup_{f \in C_b ([0,\infty))}\biggl\{ \int_0^k f(x) \omega(\d x) - \log \int_0^k e^{f(x)} \lambda_b(\d x)\biggr\}\\
	&\geq \sup_{f \in C_b ([0,k])}\biggl\{ \int f(x) \omega(\d x) - \log \int e^{f(x)} \lambda_b(\d x)\biggr\}\\
	& \ge u_k(\omega)\ge u_{k,j}(\omega) \qquad \text{ for all } j\in\bN, k\ge 0 \,.
\end{split}\]
%The second equality is due to Lusin's theorem on sets of finite measure and in the third equality, we choose a countable dense subset $(f_i)_{i\in \bN}$ of $C([0,k])$ by the  Weierstrass approximation theorem. Define
%\[u_{k,j}(\omega)= \sup_{1\le i\le j} \biggl\{ \int f_i(x) \omega(\d x) - \log \int e^{f_i(x)} \lambda_b(\d x)\biggr\}.\]
By construction and Lusin's theorem on sets of finite measure, we can approximate every of the three inequalities above arbitrary well and have the identity $H(\omega|\lambda_b)= \sup_{k\geq 0}\sup_{j\in \bbN} u_{k,j}(\omega)$.

We note, that $\omega \mapsto \int f(x) \omega(\d x) - \log \int e^{f(x)} \lambda_b(\d x) $ is continuous for each $f \in C_b([0,k])$ so that~$u_{k,j}$, as the supremum of a finite set of continuous function, is also continuous in the topology weak convergence of probability measures.

With this preliminary considerations, we can use the weak convergence of $\bP^{N,L,\epsilon} \to \bP$ and once more the dual representation for the relative entropy, to estimate
\[\begin{split}
	%\bP{\pi_t}^{-1}[u_{k,j}]&=
	\bE_t[ u_{k,j}] &= \lim_{\frac{N\epsilon}{L}\to\rho} \bE_t^{N,L,\epsilon}[u_{k,j}] = \lim_{\frac{N\epsilon}{L}\to\rho}\sum_{c\in \hat{V}^{N,L,\epsilon}} u_{k,j} (c) C^{L}_{\#}\mu_{N,L,\epsilon}(t,c)\\
&= \lim_{\frac{N\epsilon}{L}\to\rho} \sum_{\eta\in V^{N,L,\epsilon}} u_{k,j} \bigl(C^L(\eta)\bigr) % f_{N,L,\epsilon}(t, \eta)
\pderiv{\mu_{N,L,\epsilon}(t,C^L(\eta))}{\nu_{L,\epsilon,b}(\eta)} \nu_{L,\epsilon,b}(\eta)\\
& \le \lim_{\frac{N\epsilon}{L}\to\rho} \biggl( \frac{1}{L} \log \sum_{\eta\in V^{N,L,\epsilon}}  \exp\Bigl(L u_{k,j}\bigl(C^L(\eta)\bigr)\Bigr)  \nu_{L,\epsilon,b}(\eta)   + \frac1L H\bigl(\mu_{N,L,\epsilon}(t) | \nu_{L,\epsilon,b}\bigr) \biggr) \\
&\le \sup_{\omega \in \cP(\bR_+)}   \bigl\{u_{k,j}(\omega) - H(\omega | \lambda_b) \bigr\} + C t+ \limsup_{\frac{N\epsilon}{L} \to \rho}\frac1L H\bigl(\mu_{N,L,\epsilon}(0) | \nu_{L,\epsilon,b} \bigr). \end{split}\] for almost all $t\geq 0$, where we used \eqref{eq:Sanov} and the estimate~\eqref{eq:entropy-growth} from Lemma~\ref{lem:entropy:bound} in the last inequality.
%Again we use a smooth truncation of $f\in C_b$ to $[0,k)$ to get  $(f_k)_k$ approximating continuous bounded functions by continuous functions with compact support and by the dominated convergence to have the continuity of the integrals. 
Hence, we conclude by monotone convergence the estimate
\[\bE_t[ u_{k}]=  \lim_{j \to \infty}  \bE_t[u_{k,j}] \le  C(1+ t) \,. \] 
We conclude by letting $\epsilon>0$ and use the definition of supremum to argue that for each $f \in C_b([0,\infty))$, we have by Fatou's lemma the bound
\[H(\omega|\lambda_b) \le \liminf_{k\to\infty} u_k(\omega) + \epsilon,\] 
Hence, we conclude the claimed estimate~\eqref{eq:entropy:bound} and so the proof.
\end{proof}

\begin{lemma}[Construction of initial measure]\label{lemma:density}
	Suppose $\hat{V}^{N,L,\epsilon} \ni c^{N,L,\epsilon} \to c_0$ in $W_1$ for some $c_0\in \calP_\rho(\bbR_+)$ satisfies for some $\kappa>0$ the bound
	\begin{equation*}%\label{eq:initial-condition}
		\liminf_{\frac{N\epsilon}{L} \to \rho} \Bigl( \epsilon \bigl|{C^L}^{-1}(c^{N,L,\epsilon})\bigr|^{1/L}\Bigr) \geq \kappa >0  \,,
		\quad\text{where}\quad
		 |{C^L}^{-1}(c^{N,L,\epsilon})|= \frac{L!}{\prod_{k=0}^N (c^{N,L,\epsilon}({k\epsilon})L)!} \,,
	\end{equation*}
	then there exists $\tilde \mu_{N,L,\eps} \in \calP(V^{N,L,\eps})$ such that $C^L_\#\tilde \mu_{N,L,\epsilon}= \delta_{c^{N,L,\epsilon}} \to \delta_{c_0}$ and for any $b>0$ it holds
	\[
		\limsup_{N\epsilon/L\to \rho} \frac1L H\bigl(\tilde \mu_{N,L,\epsilon}| \nu_{L,\epsilon,b}\bigr) <+\infty \,.
	\]
\end{lemma}
\begin{proof}
Since $\hat{V}^{N,L,\epsilon} \ni c^{N,L,\epsilon} \to c_0$ in $W_1$, we get $\delta_{c^{N,L,\epsilon}} \to \delta_{c_0}$ weakly on $\cP(\cP^1)$. Let $\tilde \mu_{N,L,\epsilon}$ be the uniform probability measure on ${C^L}^{-1}(c^{N,L,\epsilon})= \{\eta\in V^{N,L,\epsilon} : C^L(\eta) = c^{N,L,\epsilon}\}$, 
\[\tilde\mu_{N,L,\epsilon}(\eta) = |{C^L}^{-1}(c^{N,L,\epsilon})|^{-1} \1_{{C^L}^{-1}(c^{N,L,\epsilon})} (\eta), \quad \text{ where } |{C^L}^{-1}(c^{N,L,\epsilon})|= \frac{L!}{\prod_{k=0}^N (c^{N,L,\epsilon}({k\epsilon})L)!} \] By definition, $C^L_{\#}\tilde\mu_{N,L,\epsilon}= \delta_{c^{N,L,\epsilon}}$ and for $\eta \in {C^{L}}^{-1}(c^{N,L,\epsilon})$, $\nu_{L,\epsilon,b}(\eta)= \prod_{i\in\bN_0} g_{\epsilon,b}(i \epsilon)^{c^{N,L,\epsilon}(i\epsilon) L}$. We have
\begin{equation*}\begin{split}\frac1L H(\tilde\mu_{N,L,\epsilon}| \nu_{L,\epsilon,b}) &=\frac1L\sum_{\eta \in {C^{L}}^{-1}(c^{N,L,\epsilon})} \tilde\mu_{N,L,\epsilon}(\eta) \log\frac{\tilde\mu_{N,L,\epsilon}(\eta)}{\nu_{L,\epsilon,b}(\eta)}\\
& =-\frac1L \log \Bigl(\bigl|{C^{L}}^{-1}(c^{N,L,\epsilon})\bigr| \prod_{i\in\bN_0} g_{\epsilon,b}(i\epsilon)^{c^{N,L,\epsilon}(i\epsilon) L} \Bigr)\\
& =-\log \bigl|{C^{L}}^{-1}(c^{N,L,\epsilon})\bigr|^{\frac1L }  -  \sum_{i=0}^N   c^{N,L,\epsilon}(i\epsilon) \log g_{\epsilon,b}(i \epsilon)\\
&= -\log \bigl|{C^{L}}^{-1}(c^{N,L,\epsilon})\bigr|^{\frac1L } - \sum_{i=0}^N c^{N,L,\epsilon}(i\epsilon) \log \bigl(b \epsilon e^{- i \epsilon}\bigr)  + o(1)_{\epsilon \to 0}\\
&=- \log \epsilon\bigl|{C^{L}}^{-1}(c^{N,L,\epsilon})\bigr|^{\frac1L }   -\log b  + \frac{N\epsilon}{L}  + o(1)_{\epsilon \to 0}.\end{split}\end{equation*}
Hence, we conclude
\[
\limsup_{N\epsilon/L\to\rho}\frac1L H(\tilde\mu_{N,L,\epsilon}| \nu_{L,\epsilon,b}) \le -\log \kappa -\log b + \rho < +\infty.
\]

%If $c^{N,L,\epsilon}$ is distributed uniformly over $h(L)$ points with centre of mass at $\frac{N\epsilon}{L}$  where  $h\to \infty$ and at most linear in $L$, then we have  
%\[\log L! - \sum_{i=0}^N \log (c(i \epsilon) L)! = L(\log h(L)+o(1)) \] so that in this case 
%\[H\bigl(\tilde\mu_{N,L,\epsilon}| \nu_{L,\epsilon,b}\bigr)  = -\log\bigl(b\epsilon h(L)\bigr) + O(1).\] and for this term to be bounded we need $\epsilon \approx \frac{1}{h(L)}$ which makes sense as for the sequence of measures to convergence to a measure with Lebesgue density,  the mass in an interval of $\bR_+$ should be bounded along the sequence of discrete measures. 
\end{proof}
\begin{remark}
Note that
 \[\log L + O(1) =L^{-1}\log L!  \ge \log|{C^{L}}^{-1}(c^{N,L,\epsilon})\bigr|^{\frac1L } \ge 0.\]  
In the case $\frac{N}{L} \in \bN$, the zero is attained if  $c^{N,L,\epsilon}$ is concentrated at $\frac{N}{L}\epsilon$, where we use the convention $0!=1$, and the maximum $ \log L!$ is attained  if $c^{N,L,\epsilon}$ is evenly distributed over $L$ points  with mass $1/L$ and the center of mass is at $\frac{N}{L} \epsilon$. In this way, this quantity measures the spread of the measure $c^{N,L,\epsilon}$.
\end{remark}
\begin{corollary}\label{cor:densitfy:sufficient}
Let $c\in \cP_{\rho}(\bR_+)$  with uniformly continuous and uniformly bounded Lebesgue density. Then there exists $ \hat{V}^{N,L,\epsilon}\ni c^{N,L,\epsilon}  \to c$ in $W_1$ as $\frac{N\epsilon}{L}\to \rho$ with 
\begin{equation}\label{eq:densitfy:sufficient}
	\liminf_{\frac{N\epsilon}{L} \to \rho} \Bigl( \epsilon \bigl|{C^L}^{-1}(c^{N,L,\epsilon})\bigr|^{1/L}\Bigr) \geq \|c\|_\infty^{-1}>0 \,,
\end{equation}  
where $\|c\|_{\infty}$ is the supremum norm of the density of $c$.
\end{corollary}
\begin{proof}
Let $c^{L,\epsilon}(k\epsilon)= \frac1L \bigl\lfloor L \int_{k\epsilon}^{(k+1)\epsilon} c(x) \d x  \bigr\rfloor$ for $k\in\bN_0$.  Define $a_{L,\epsilon} = 1 - \sum_{k\in\bN_0}c^{L,\epsilon}(k\epsilon) \in\frac1L \bN_0$ and $b_{L,\epsilon} =L (\frac{N}{L} -\sum_{k\in\bN_0}kc^{L,\epsilon}(k\epsilon) ) \in\bN_0$.  In the case for  $a_{L,\epsilon}\geq \frac{1}{L}$, we define $c^{N,L,\epsilon}=c^{L,\epsilon}+ \frac{1}{L} \delta_{b_{L,\epsilon}\epsilon}+\bigl(a_{L,\eps}-\frac1{L}\bigr)\delta_0$. Otherwise if $a_{L,\epsilon}<\frac{1}{L}$, hence $a_{L,\eps}=0$ and so $c^{L,\epsilon}\in \cP(\bN_0\epsilon)$, 
we define $c^{N,L,\epsilon} = c^{L,\epsilon} - \frac{1}{L}\delta_{k'\epsilon}+\frac{1}{L}\delta_{(k'+b_{k,\epsilon})\epsilon}$ where $k' =\inf\{k: c^{L,\epsilon}(k \epsilon)>0\}$. Then $c^{N,L,\epsilon} \in \hat{V}^{N,L,\epsilon}$ and $c^{N,L,\epsilon} \to c$ weakly. Since $\cM_1(c^{N,L,\epsilon})=N\epsilon/L \to \rho=\cM_1(c) $, we also have $c^{N,L,\epsilon} \to c$ in~$W_1$.
 
We want to show that the constructed sequence $(c^{N,L,\eps})$ satisfies~\eqref{eq:densitfy:sufficient}
%\[ \liminf_{N\epsilon/L \to \rho}\epsilon \bigl|{C^L}^{-1}(c^{N,L,\epsilon})\bigr|^{1/L} \ge \|c\|_\infty^{-1}\]
by showing
\begin{align}\label{eq:log-c-bdd}
\liminf_{N\epsilon/L \to \rho} \log \biggl(\epsilon\bigl|{C^L}^{-1}(c^{N,L,\epsilon})\bigr|^{1/L} \biggr)\ge \log \|c\|_\infty^{-1}.
 \end{align}
By Stirling's formula,  we have 
\begin{align*}\frac1L \biggl( \log L!& - \sum_{k=0}^N \log (c^{N,L,\epsilon}(k\epsilon)L)! \biggr) =  \sum_{k=0}^N c^{N,L,\epsilon}(k\epsilon)  \bigl( \log L -\log (c^{N,L,\epsilon}(k\epsilon)L)\bigr) + O(L^{-1}\log L)\\
&=- \sum_{k=0}^N c^{N,L,\epsilon}(k\epsilon)  \log (c^{N,L,\epsilon}(k\epsilon)) + O(L^{-1}\log L). \end{align*}
The uniform continuity of $c$ implies the estimate 
\[c^{L,\epsilon}(k\epsilon)\le \int_{k\epsilon}^{(k+1)\epsilon} c(x) \d x = c(k\epsilon)\epsilon + o(\epsilon).\] With the uniform boundedness of $c$, we obtain 
\[\|c^{N,L,\epsilon}\|_\infty \le \|c^{L,\epsilon}\|_\infty \le \|c\|_\infty \epsilon + o(\epsilon).\] 
Hence, we conclude
\begin{equation*}
	-\sum_{k=0}^N c^{N,L,\epsilon}(k\epsilon)  \log (c^{N,L,\epsilon}(k\epsilon))
	\ge  \log (\|c\|_\infty \epsilon + o(\epsilon))^{-1}  \,.	
\end{equation*}  
 Applying the lower bound to $\log  \epsilon\bigl|{C^L}^{-1}(c^{N,L,\epsilon})\bigr|^{1/L} $ and taking the $\liminf $ of the expression, we obtain ~\eqref{eq:log-c-bdd}.
\end{proof}
\section{Uniqueness of weak solutions}

In this section, we prove Theorem~\ref{thm:unique}. 
The proof is an adaptation of the uniqueness arguments for the EDG in~\cite{Schlichting2019}.

In this section, we use the physics convention to write the differential right after the integral sign in front of the integrand, that is
\begin{equation*}
	\int_0^\infty \int_0^\infty f(x,y) \dx{y} \dx{x} = \int_0^\infty \dx{x} \int_0^\infty \dx{y} \, f(x,y) \,.
\end{equation*}
\begin{proof}[Proof of  Theorem~\ref{thm:unique}]
Suppose $c, d$ are two solutions of \eqref{eq:weak-sol} having Lebesgue densities and the same initial conditions. 
In the following, we identify the Lebesgue densities with the measures. 
For $k \in \bR_+$, we define
\[E(k)= \int_k^\infty \d x \, (c(x) - d(x))
\] so that
$W_1(c,d) = \int_0^\infty \d k |E(k)|$. 
By an approximation, we can find Lipschitz functions $f_n \to \1_{[k,\infty)}$ pointwise and in~$L^1$, which allows us to use $\1_{[k,\infty)}$ as a test function to get the identity
\[ \frac{\d}{\d t} E_t(k) = \int_0^\infty \d l \,\1_{[k,\infty)} (l) (Qc_t(l) - Qd_t(l)).\] 
Then, recalling the definition~\eqref{eq:def:generator} of the generator, we get 
\[\begin{split}
	Qc&=  \int_0^\infty\d x  \int_0^\infty \d y   \int_0^x \d z\,c(x)c(y) K(x,y,z) \gamma^{x,y,z}\\
	&= \int_0^\infty\d z  \int_0^\infty \d y   \int_z^\infty \d x \,c(x)c(y) K(x,y,z) \gamma^{x,y,z}
\end{split}\] 
and 
\[\begin{split} \int_0^\infty \d l \; \1_{[k,\infty)} (l) Qc(l)
=  & \int_0^\infty\d z  \int_0^\infty \d y  \int_z^\infty \d x  \; c(x)c(y)K(x,y,z)\\
&\qquad \bigl(-\1_{[k,\infty)}(x) - \1_{[k,\infty)}(y)  + \1_{[k,\infty)}(x-z)+ \1_{[k,\infty)}(y+z)\bigr).
\end{split}\]
We introduce the abbreviations
\[
A[c](y,z)= \int_z^\infty \d x \,c(x) K(x,y,z) %= \int_0^\infty \d x \,c(x) K(x,y,z)
\qquad\text{and}\qquad
B[c](x,z)= \int_0^\infty \d y \,c(y) K(x,y,z) \,,
\]
as well as $J[c](y,z)= A[c](y,z)c(y)- B[c](y+z,z)c(y+z)$. 
Then, by combining the first and third as well as second and fourth indicator, respectively and after a change of variables, we get
\[\begin{split}
\int_0^\infty \d l \,\1_{[k,\infty)}(l) Qc(l) &=   \int_0^\infty \d z   \biggl(-\int_{\max(z,k)}^{z+k} \d x \,c(x)  B[c](x,z)+ \int_{\max(k-z,0)}^{k} \d  y \,c(y)   A[c](y,z)\biggr) \\
& = \int_0^\infty \d z \int_{\max(0,k-z)}^k \d y \,\bigr(c(y) A[c](y,z) - c(y+z) B[c](y+z,z)\bigr)  \\
%&= - \int_0^\infty \d z \int_{\max(k,z)}^{z+k} \d x c(x) B[c](x,z) + \int_0^\infty \d x \int_0^x \d z \int^k_{\max(k-z,0)} \d y c(y)c(x) K(x,y,z)\\
%&= - \int_0^\infty \d z \int_{\max(k,z)}^{z+k} \d x c(x) B[c](x,z) + \int_0^\infty \d z  \int^k_{\max(k-z,0)} \d y c(y) A[c](y,z)\\
%&=\int_0^\infty \d z (-\int_{\max(k-z,0)}^{k} \d x \, c(x+z)B[c](x+z,z)+ \int_{\max(k-z,0)}^{k} \d y \,c(y)A[c](y,z) )\\
&=\int_0^\infty \d z \int_{\max(0,k-z)}^k \d y \,J[c](y,z). 
 \end{split}\]
Note that $e(y)= c(y)-d(y) = -E'(y)$. By integration by parts,
 \[\begin{split}
  &\frac{\d}{\d t} E(k)  = \int_0^\infty \d z \int_{\max(k-z,0)}^k \d y \,(J[c](y,z)- J[d](y,z)) \\
&= \int_0^\infty \d z \int_{\max(k-z,0)}^k \d y \Bigl( A[e](y,z) c(y) - B[e](y+z,z) c(y+z) \\
&\hskip12em + A[d](y,z) e(y) - B[d](y+z,z) e(y+z)\Bigr) \\
&= \int_0^\infty \d z \int_{\max(k-z,0)}^k \d y  \Bigl( A[e](y,z) c(y) - B[e](y+z,z) c(y+z) \Bigr) \\
&\quad +\int_0^\infty \d z  \int_{\max(k-z,0)}^k \d y \Bigl( \partial_1 A[d](y,z) E(y) - \partial_1 B[d](y+z,z) E(y+z) \Bigr) \\
& \quad +\int_0^\infty \d z  \Bigl(- A[d](y,z)E(y)+ B[d](y+z,z) E(y+z) \Bigr)\Big\vert^{y=k}_{y=\max(k-z,0)}.
\end{split}\]
Noting that $E(0)=\int_0^\infty \d k \, c(k) - \int_0^\infty \d k \, d(k) = 1-1=0$ and $B[d](z,z)=0$ from Assumption~\ref{ass:unique}, the last integral is then simplified to
\[\begin{split} 
	\MoveEqLeft \int_0^\infty \d z  \Bigl(- A[d](y,z)E(y)+ B[d](y+z,z) E(y+z) \Bigr)\Big\vert^{y=k}_{y=\max(k-z,0)}\\
&= \int_0^\infty \d z \Bigl( - A[d](k,z)E(k)+ B[d](k+z,z) E(k+z)\Bigr) \\
	&\quad + \int_0^k \d z \Big(A[d](k-z,z)E(k-z) - B[d](k,z) E(k)\Big).
\end{split}\]
By another integration by parts, we get
\[\begin{split}
	A[e](y,z) &= \int_0^\infty \d x \,e(x) K(x,y,z) = -\int_0^\infty \d x \,E'(x) K(x,y,z)\\
& = - E(x)K(x,y,z)\big\vert^{x=+\infty}_{x=0} + \int_0^\infty \d x \,E(x) \partial_1 K(x,y,z) =\int_0^\infty \d x \,E(x) \partial_1 K(x,y,z),
\end{split}\]
where the boundary term vanishes because by the boundedness of the first and zeroth moments, it holds
\[ (1+x)  \int_x^\infty \d k \,c(k) \le \int_x^\infty \d k \,c(k)  (1+k) \to 0 
\quad\text{ as }\quad  x\to +\infty \]
and $E(0)= 0.$ Similarly, we find 
\[ B[e](y,z) = \int_0^\infty \d x \,E(x) \partial_2 K(y,x,z) \,. \]
In summary, we arrive at 
\[\begin{split}\frac{\d |E(k)|}{\d t}&= \operatorname{sign} (E(k)) \frac{\d E(k)}{\d t} \\
&\le \int_0^\infty \d z \int_{\max(k-z,0)}^k \d y  \Bigl( |A[e](y,z)|c(y) + |B[e](y+z,z)|c(y+z) \\
&\hskip12em + |\partial_1 A[d](y,z)| |E(y)|+ |\partial_1 B[d](y+z,z)| |E(y+z)| \Bigr) \\
&\quad + \int_0^\infty \d z \Bigl( - |E(k)| A[d](k,z) + |E(k+z)| B[d](k+z,z) \Bigr) \\
&\quad +\int_0^k \d z \Bigl( |E(k-z)| A[d](k-z,z)  - |E(k)|B[d](k,z) \Bigr) .
\end{split}\]
To estimate $\int_0^\infty \d k  \frac{\d |E(k)|}{\d t}$, we begin with using the assumption \eqref{ass: 1-diff K},  $|\partial_1  K(x,y,z)|\le(1+y) \varphi(z)$. 
Moreover, we use that for $\varphi \ge 0, \varphi\in L^{1,1}$, it holds
\[ 
	\int_0^\infty \d k \int_{k}^\infty \d z \, \varphi(z) = \int_0^\infty \d z \, \varphi(z)  z,
	\qquad\text{ and }\qquad
	y+z - \max(z,y) \le z \,,
\]  
as well as the conversation of mass and the first moment for $c\in \mathcal{P}_\rho$, it holds
\[
	\int_0^\infty (1+y)c(y) \d y = 1+ \cM_1(c) = 1+ \rho \,.
\]
With this preliminary considerations and by carefully interchanging the integrals keeping track of the integration domains, we get
\[\begin{split}
	\MoveEqLeft \int_0^\infty \d k \int_0^\infty \d z \int_{\max(k-z,0)}^k \d y \, 
	\bigl\lvert A[e](y,z)\bigr\rvert c(y) \\
	&\le \int_0^\infty \d k  \int_0^\infty \d z \int_{\max(k-z,0)}^k \d y  \int_0^\infty \d x \, \bigl\lvert E(x)\bigr\rvert \, \bigl\lvert \partial_1 K(x,y,z)\bigr\rvert c(y)\\
& \le \int_0^\infty \d x \, \bigl\lvert E(x)\bigr\rvert \int_0^\infty \d k
 \biggl[ \int_0^k \d z \, \varphi(z) \int_{k-z}^k \d y \, (1+y) c(y)+ \int_k^\infty \d z \, \varphi(z) \int_{0}^k \d y\, (1+y)  c(y) \biggr] \\
& \le \int_0^\infty \d x \,\bigl\lvert E(x)\bigr\rvert   \biggl[ \int_0^\infty \d z \, \varphi(z) \int_0^\infty  \d y \,(1+y) c(y) \int^{y+z}_{\max(z,y)} \d k + \int_0^\infty \d k \int_k^\infty \d z\,\varphi(z) (1+\rho) \biggr]  \\
& \le  \int_0^\infty \d x \,\bigl\lvert E(x) \bigr\rvert \biggl[ 2 \int_0^\infty \d z \, \varphi(z) z  (1+\rho) \biggr] \,.\\
 \end{split}\]
%\[\sup_{c\in \cP_\rho }\int_0^\infty \d k  \int_0^k \d z \varphi(z) \int_{k-z}^k \d y  f_1(y) c(y) < C\]
%$|\partial_2 K(x,y,z)|\le f_2(x,z)$
%\[\sup_{c \in \cP_\rho} \int_0^\infty \d k \int_0^\infty \d z \int_k^{k+z} \d x  f_2(x,z) c(x) < C\]
%Examples: for $i =1,2$\begin{itemize}
%\item support of $(f_i(\cdot,z))_{z\ge 0}$ is contained in a compact set and is bounded,
%\item $(f_i(y, \cdot)= f_i(y) \varphi(\cdot))_y$, $\varphi$ has  compact support, $f_i$ vanish at infinty and bounded with sufficient decay.
%\end{itemize} }
By using the assumption $|\partial_2 K(x,y,z)|\le (1+x) \varphi(z)$ and similar arguments as for $A[e]$ above, we also have 
\[\begin{split}\int_0^\infty \d k \int_0^\infty \d z \int_{\max(k-z,0)}^k \d y \, \bigl\lvert B[e](y+z,z) \bigr\rvert c(y+z)
%	=\int_0^\infty \d k \int_0^\infty \d z \int_{\max(k,z)}^{k+z} \d y |B[e](y,z)|c(y) \\
%& \le \int_0^\infty \d k \int_0^\infty \d z \int_{\max(k,z)}^{k+z} \d y \int_0^\infty \d x |E(x)| |\partial_2 K(y,x,z)| c(y)\\
%& \le \int_0^\infty\d x |E(x)|   \int_0^\infty \d k \int_0^\infty \d z \varphi(z) \int_{\max(k,z)}^{k+z} \d y (1+y) c(y)\\
%& \le \int_0^\infty\d x |E(x)| \Big(  \int_0^\infty \d k \int_k^\infty  \d z  \varphi(z)    \int_{z}^{k+z} \d y (1+y) c(y) + \int_0^\infty \d z   \varphi(z) \int_z^\infty \d k \int_{k}^{k+z} \d y (1+y) c(y)\Big)\\
% &\le   \int_0^\infty\d x |E(x)| \Big(  \int_0^\infty \d k \int_k^\infty \d z \varphi(z) (1+\rho) + \int_0^\infty \d z   \varphi(z) \int_z^\infty \d y (1+y) c(y) \int_{\max(z,y-z)}^y \d k \Big)\\
& \le \int_0^\infty\d x \, \bigl\lvert E(x) \bigr\rvert \biggl[ 2 \int_0^\infty \d z \, \varphi(z) z  (1+\rho)\biggr] \end{split}.\]
We recall the assumption~\eqref{ass: 1-diff K}: $|\partial_1 \partial_2 K(x,y,z)|\le \varphi(z)$ and define $D(x)= \int_x^\infty \d y \, d(y)$ to estimate
\[\begin{split}
	\MoveEqLeft \int_0^\infty \d k \int_0^\infty \d z  \int_{\max(k-z,0)}^k \d y  \, \bigl\lvert E(y)\bigr\rvert \, \bigl\lvert\partial_1 A[d](y,z)\bigr\rvert \\
 &\le \int_0^\infty \d k \int_0^\infty \d z \int_{\max(k-z,0)}^k \d y \, \bigl\lvert E(y)\bigr\rvert \int_0^\infty \d x \, D(x) \bigl\lvert \partial_2 \partial_1 K(x,y,z) \bigr\rvert \\
& \le \|D\|_1 \int_0^\infty \d k \biggl[ \int_0^{k} \d z \, \varphi(z) \int_{k-z}^k \d y \, \bigl\lvert E(y)\bigr\rvert    + \int_{k}^{\infty} \d z \, \varphi(z) \int_{0}^k \d y \, \bigl\lvert E(y)\bigr\rvert \biggr] \\
& \le  \|D\|_1 \biggl[ \int_0^\infty \d z \, \varphi(z)  \int_0^{\infty} \d y \, \bigl\lvert E(y)\bigr\rvert  \int_{\max(z,y)}^{y+z} \d k  +   \int_0^\infty \d k \int_{k}^{\infty} \d z \, \varphi(z) \int_{0}^k \d y \, \bigl\lvert E(y)\bigr\rvert \biggr]\\
& \le  \|D\|_1  \int_0^\infty \d y \, \bigl\lvert E(y)\bigr\rvert \biggl[ \int_0^\infty \d z \, \varphi(z) z    + \int_0^\infty \d z \, \varphi(z) z   \biggr] \\
& \le 2  \|D\|_1  \int_0^\infty \d y \, \bigl\lvert E(y)\bigr\rvert \biggl[ \int_0^\infty \d z \, \varphi(z)z \biggr] \,. \end{split}\]
Similarly, we get the estimate 
\[\begin{split}
	\int_0^\infty \d k \int_0^\infty \d z \int_{\max(k-z,0)}^k \!\! \d y \,  \bigl\lvert E(y+z)\bigr\rvert \bigl\lvert \partial_1 B[d](y+z,z)\bigr\rvert 
% &= \int_0^\infty \d k \int_0^\infty \d z \int_{\max(k,z)}^{k+z} \d y |E(y)||\partial_1 B[d](y,z)| 
%\\
%& \le  \int_0^\infty \d k \int_0^\infty \d z \int_{\max(k,z)}^{k+z} \d y |E(y)|\int_0^\infty \d x D(x)|\partial_1\partial_2 K(y,x,z)| \\
%&\le   \|D\|_1  \int_0^\infty \d k \int_0^\infty \d z \, \varphi(z) \int_{\max(k,z)}^{k+z} \d y  |E(y)|   \\
%& \le  \|D\|_1 \Big( \int_0^\infty \d k  \int_0^k \d z   \varphi(z) \int_{k}^{k+z} \d y  |E(y)|   + \int_{0}^{\infty} \d y  |E(y)| \int_0^\infty \d k  \int_k^\infty \d z   \varphi(z) \Big)\\
%& \le  \|D\|_1 \Big(   \int_0^\infty \d z   \varphi(z)  \int_z^\infty \d y |E(y)| \int_{\max(z,y-z)}^y \d k 
%+  \int_{0}^{\infty} \d y  |E(y)|  \int_0^\infty \d z \,z \varphi(z) \\
%&\le  \|D\|_1 \Big(   \int_0^\infty \d z   \varphi(z)  z  \int_z^\infty \d y |E(y)| 
%+  \int_{0}^{\infty} \d y  |E(y)|  \int_0^\infty \d z \,z \varphi(z) \\
& \le 2 \|D\|_1  \int_0^\infty \d y \bigl\lvert E(y)\bigr\rvert \biggl[ \int_0^\infty \d z  \,  \varphi(z)  z \biggr] \,. 
\end{split}\]
Since the first moment is preserved under time evolution, the four terms above can be controlled by $\cM_1(c_0)=\cM_1(d_0)= \rho=\|D\|_1$.
Next, with a change of variable,
%and recall the convention that takes the function to be zero whenever the arguments are negative and $z\ge x$ in $K(x,y,z)$, 
we observe that the boundary terms vanish
\[\begin{split}\int_0^\infty \d k  \int_0^k \d z \, \bigl\lvert E(k-z)\bigr\rvert A[d](k-z,z) &= \int_0^\infty \d z \int_z^\infty \d k  \, \bigl\lvert E(k-z)\bigr\rvert A[d](k-z,z) \\
& = \int_0^\infty \d k \int_0^\infty \d z \, \bigl\lvert E(k)\bigr\rvert A[d](k,z). 
\end{split}\]
as well as
\[\begin{split}
\int_0^\infty \d k \int_0^\infty \d z \, \bigl\lvert E(k+z) \bigr\rvert B[d](k+z,z)& = \int_0^\infty \d z \int_z^\infty \d k \, \bigl\lvert E(k) \bigr\rvert B[d](k,z)\\
& =\int_0^\infty \d k \int_0^k \d z \, \bigl\lvert E(k) \bigr\rvert B[d](k,z) \,.
\end{split}\]
In summary, we conclude the Gronwall estimate
\[\frac{\d }{\d t} \int_0^\infty \d k \, \bigl\lvert E_t(k) \bigr\rvert \le C \int_0^\infty \d x \bigl\lvert E_t(x) \bigr\rvert \,, \] 
where $C = 4\|\varphi\|_{1,1} (1+2 \rho) $
so that by the Gronwall Lemma, we conclude
\[
	\int_0^\infty \d k \, \bigl\lvert E_t(k)\bigr\rvert  \le e^{C t} \int_0^\infty \d k \, \bigl\lvert E_0(k)\bigr\rvert  = 0 \,, 
\]
since $E_0(k)=0$ for each $k\geq 0$ given the same initial condition. 
\end{proof}

%%%%%%%%%%%%%%%%%%
%%% REFERENCES %%%
%%%%%%%%%%%%%%%%%%
\bibliographystyle{abbrv}
\bibliography{bib.bib}
\end{document}